\pdfoutput=1
\documentclass[a4paper, reqno, 12pt, toc]{amsart}
\usepackage[utf8]{inputenc}
\usepackage[T1]{fontenc}
\usepackage{lmodern, csquotes, graphicx}
\usepackage{amssymb, physics, bm}
\usepackage{hyperref}
\usepackage[subsetnonamb, bbsets]{jkmath}
\usepackage[backend=biber,style=alphabetic,sorting=nyt]{biblatex}
\usepackage{dynkin-diagrams}

\newtheorem{theorem}{Theorem}[section]
\newtheorem{definition}[theorem]{Definition}
\newtheorem{lemma}[theorem]{Lemma}
\newtheorem{proposition}[theorem]{Proposition}
\newtheorem{corollary}[theorem]{Corollary}

\newtheorem{notation}[theorem]{Notation}
\newtheorem{remark}[theorem]{Remark}
\newtheorem{conjecture}[theorem]{Conjecture}

\DeclareMathOperator{\id}{id}

\DeclareMathOperator{\ad}{ad}

\DeclareMathOperator{\Hom}{Hom}
\DeclareMathOperator{\lot}{l.o.t.}
\DeclareMathOperator{\ct}{ct}

\DeclareMathOperator{\diag}{diag}
\DeclareMathOperator{\Rad}{Rad}
\DeclareMathOperator{\gr}{gr}
\DeclareMathOperator{\wt}{wt}
\newcommand{\uq}{\mathbf {U}}

\newcommand{\uqb}{\mathbf{B}}

\newcommand{\adr}{\mathrm{ad}_{\mathrm{r}}}

\begin{document}
\title[ZSF of Quantum Symmetric Pairs]{Zonal Spherical Functions of Quantum Symmetric Pairs}

\author[P. Schlösser]{Philip Schlösser}
\address{IMAPP-Mathematics, PO Box 9010, 6500 GL Nijmegen, The Netherlands}
\email{philip.schloesser@ru.nl}
\urladdr{https://philip98.github.io}
\subjclass[2020]{Primary 33D80; Secondary 33D52, 17B37}
\keywords{Quantum symmetric pair, zonal spherical function, character spherical
function, Macdonald--Koornwinder polynomial}

\thanks{My thanks to Stein Meereboer whom I could always bounce ideas off of. This work is funded by grant \texttt{OCENW.M20.108} of the
Dutch Research Council.}

\begin{abstract}
    We identify the zonal and character spherical functions for quantum symmetric pairs with the symmetric Koornwinder--Macdonald polynomials.
    To this end, the methods of Letzter's 2004 paper are translated to modern conventions and right coideal subalgebras, and extended
    to non-standard cases and cases with non-reduced restricted root systems.
    For the last elusive Satake type, $\mathsf{FII}$, a conjecture is provided.
\end{abstract}

\maketitle

\section{Introduction}
The study of quantum symmetric pairs (QSP) and their spherical functions has advanced vastly in conventions, generality, and knowledge since
one of its most foundational papers, Gail Letzter's \cite{Le04}, was written.
This paper's aim is to translate the contents of loc.cit. into modern
notation, provide more details in the proofs and a better
understanding of the techniques that Letzter used, and also deliver
on Letzter's promise from loc.cit.'s introduction: show that these
techniques work in the general case as well.
We then use these techniques and combine them with results on rank-1 cases from the literature to identify the zonal spherical functions
of QSPs with twisted Macdonald--Koornwinder polynomials for all finite-type QSPs except for
type $\mathsf{FII}$ for which we provide a conjecture.
This happens in Theorem~\ref{thm-zsf} (and Conjecture~\ref{conj-FII}).

In particular, we determine the zonal spherical functions for all
Hermitean QSPs, i.e. those $(\uq,\uqb)$ where
$\uqb$ has characters besides the counit $\epsilon$.
This allows us to extend the results from \cite{Mee25} to
non-standard parameters, to different parameters on the left and the right, and to different characters on the left and the right
(Theorem~\ref{thm-chisf}).
For specific choices of characters, the resulting functions can also
be interpreted as $\epsilon$-symmetric Macdonald--Koornwinder polynomials (Proposition~\ref{prop-epsilon-koornwinder}).

Lastly, we provide a survey about which Koornwinder polynomials
(i.e. Macdonald for a $(\mathsf{C}_n^\vee, \mathsf{C}_n)$ affine root system)
can be realised as character spherical functions of a quantum symmetric
pair.

\subsection{Structure}
In Section~\ref{sec-general} we introduce standard notation around
QSPs, their root systems, and their spherical functions.
In particular, we introduce the restriction procedure that produces
symmetric polynomials from spherical functions.

In Section~\ref{sec-centraliser} we study the centraliser
$Z_{\widetilde{\uq}}(\widetilde{\uqb})$ and use it to construct a
commutative algebra of symmetric difference operators (the radial parts of $Z_{\widetilde{\uq}}(\widetilde{\uqb})$) that are 
diagonalised by the restrictions of zonal spherical functions.
Our goal is to show that these symmetric difference operators
are Macdonald operators.

In Section~\ref{sec-graded} we establish the existence of a power
series $p$ that determines the
leading terms of the radial parts of $Z_{\widetilde{\uq}}(\widetilde{\uqb})$ in a uniform manner.

In Section~\ref{sec-rk1-reduction} we establish that $p$ can be
computed from the $p$-functions of rank-1 QSPs.
In Section~\ref{sec-rk1-computation} we compute these rank-1
$p$-functions from cases that are known from literature.

In Section~\ref{sec-link-macdonald} we consider small cases (that
always exist) in which the leading term of a radial part actually 
determines the entire radial part, so that we can compute these
radial parts using our knowledge of $p$.
This knowledge is sufficient to identify the joint eigenfunctions, i.e.
the restrictions of ZSF, as symmetric Macdonald--Koornwinder polynomials.

In the next section (Section~\ref{sec-characters}), we complete
the computation of character spherical functions from \cite{Mee25}
by also considering non-standard cases and by considering
different parameters and characters on the left and right.
For specific non-standard cases, this also gives us an interpretation
of character spherical functions as $\epsilon$-symmetric
Koornwinder polynomials.
For specific standard cases, this also allows us to classify the
zonal spherical functions for quantisations of several non-symmetric
spherical pairs.

Finally, in Section~\ref{sec-classification-MK}, we classify which
parameter values of rank-$n$ Koornwinder polynomials are
\enquote{geometric}, i.e. can be realised as restrictions of character
spherical functions.

\section{General Definitions and Notations}\label{sec-general}
\subsection{Root System}
Let $I:=\set{1,\dots,n}$ and let $A=(a_{ij})_{i,j=1,\dots,n}$ be a
symmetrisable Cartan matrix of finite type and let $X,Y$ be two lattices equipped with
a perfect pairing $\langle\cdot,\cdot\rangle: Y\times X\to \Z$, and
with linearly independent elements
\[
    \alpha_1,\dots,\alpha_n\in X,\qquad
    h_1,\dots,h_n\in Y
\]
such that
\[
    \langle h_i,\alpha_j\rangle = a_{ij}\qquad (i,j\in I).
\]
Let there be a symmetric bilinear form $\cdot$ on $X$ such that
\[
    \forall i,j\in I\colon\quad \langle h_i,\alpha_j\rangle = 2\frac{\alpha_i\cdot\alpha_j}{\alpha_i\cdot\alpha_i}.
\]
This corresponds to a choice $(I,\cdot)$ of Cartan datum
(\cite[\S1.1.1]{Lus10}), and a choice $(X,Y)$ of $X$- and $Y$-regular
root datum (\cite[\S2.2.1]{Lus10}), where the embeddings of
$I\to X,Y$ are denoted by $i\mapsto \alpha_i,h_i$, respectively.
Furthermore, we write $h: X\to Y\otimes_\Z\Q$ for the linear map mapping
$\lambda\in X$ to the unique element $h\in Y\otimes_\Z\Q$ satisfying
\[
    \forall\mu\in X\colon\quad\lambda\cdot\mu = \langle h,\mu\rangle.
\]
In particular, $h_{\alpha_i} = \frac{\abs{\alpha_i}^2}{2}h_i$.
Assume that $\abs{\alpha_i}^2\in 2\Z$, so that $h_{\alpha_i}\in Y$.

Write $Q,Q^\vee$ for the sublattices of $X,Y$ generated by
$\alpha_i, h_i$ ($i\in I$), respectively.
Write $P^\vee, P$ for the corresponding dual lattices (which are then
quotients of $Y, X$, respectively).

Let $W$ be the corresponding Weyl group; it acts on $X, Y$ and their
sub- and quotient lattices $Q, Q^\vee, P, P^\vee$.
Let $R=\bigcup_{i\in I}W\alpha_i\subset Q$ be the root system associated
to $A$, and let $Q^+:= \bigoplus_{i\in I} \N_0\alpha_i$ and
$R^+:= R\cap Q^+$ (similarly, $R^\vee, Q^{\vee+},R^{\vee+}$).
Write furthermore $X^+:= \set{\mu\in X\where\langle Q^{\vee+},\mu\rangle\subset\N_0}$ for the set of dominant elements of $X$.

Let $2\rho\in X$ be the sum of positive roots.

\subsection{Quantum Group}
Let $q$ be an indeterminate and consider $k:= \overline{\C(q)}$.
Let $K$ be an invertible element of a $k$-algebra $A$; for
$a,b\in\Z$ we define
\[
    [a]_{q^b} := \frac{q^{ab}-q^{-ab}}{q^b-q^{-b}},\qquad
    [K;a]_{q^b} := \frac{Kq^{ab}-K^{-1}q^{-ab}}{q^b-q^{-b}}
\]
as well as $[K]_{q^b}:= [K;0]_{q^b}$ if this does not lead to confusion.
Define $q_i := q^{\frac{\abs{\alpha_i}^2}{2}}$.

\begin{definition}
    The \emph{Drinfel'd--Jimbo quantum group} $\uq$ is the associative
    $k$-algebra generated by
    \[
        (E_i)_{i\in I},\qquad (F_i)_{i\in I},\qquad
        (K_h)_{h\in Y}
    \]
    subject to the relations
    \begin{align*}
        K_0 &= 1\\
        K_h K_{h'} &= K_{h+h'}\\
        K_h E_i &= q^{\langle h,\alpha_i\rangle} E_iK_h\\
        K_h F_i &= q^{-\langle h,\alpha_i\rangle} F_iK_h\\
        \comm{E_i}{F_j} &= \delta_{ij}[K_{\alpha_i}]_{q_i}\\
        F_{ij}(E_i, E_j) &= 0\qquad (i\ne j)\\
        F_{ij}(F_i, F_j) &= 0\qquad (i\ne j),
    \end{align*}
    where
    \[
        F_{ij}(x,y) := \sum_{p+p'=1-a_{ij}}
        \frac{(-1)^{p'}}{[p]^!_{q_i} [p']^!_{q_i}} x^p y x^{p'},
    \]
    $h,h'\in Y, i,j\in I$,
    and where $K_\mu$ is understood to mean $K_{h_\mu}$ for
    $\mu\in X$ with $h_\mu\in Y$.
    Here, we used $[p]^!_q := [p]_q[p-1]_q\cdots [1]_q$ for
    $p\in\N_0$.

    We also define the completion $\widetilde{\uq}$ of
    $\uq$ to be the algebra generated by
    \[
        (E_i)_{i\in I},\qquad (F_i)_{i\in I},\qquad
        (K_h)_{h\in Y\otimes\Q}
    \]
    subject to the same relations.

    Write $\mathcal{A}$ for the vector space of matrix elements
    of finite-dimensional representations.
\end{definition}

There is a well-established way (e.g. \cite[\S3.3.4]{Lus10} or
\cite[Lemma~3.2]{qmsf}) of turning $\uq$ and $\mathcal{A}$ into dual Hopf
algebras with comultiplication $\Delta$, counit $\epsilon$, and antipode $S$.
Sweedler notation will be used whenever appropriate.

We define $\uq^0$ (resp. $\widetilde{\uq}^0$) to be the commutative 
$k$-subalgebra of $\uq$ (resp. $\widetilde{\uq}$) generated by $K_h$
for $h\in Y$ (resp. $h\in Y\otimes\Q$).

Write $\uq^{\pm}$ for the subalgebras generated by the $E_i$ 
(resp. $F_i$) for $i\in I$.
We have
\[
    \uq\cong\uq^+\otimes\uq^0\otimes \uq^-,\qquad
    \widetilde{\uq}\cong \uq^+\otimes\widetilde{\uq}^0\otimes\uq^-.
\]
Write $\ad,\adr$ for the left- and right-adjoint actions of $\uq$
(resp. $\widetilde{\uq}$) on itself.
They are given as follows:
\[
    \ad(x)(y) := x_{(1)} y S(x_{(2)}),\qquad
    \adr(x)(y) := S(x_{(1)}) y x_{(2)}.
\]
For any $\uq$-module $V$ and $\mu\in X$ let
\[
    V_\mu := \set{v\in V\where\forall h\in Y: K_hv=q^{\langle h,\mu\rangle}v},
\]
the \emph{$\mu$-weight space}.
Applying this to the left-adjoint action of $\uq$ on itself,
we obtain a decomposition $\uq = \bigoplus_{\mu\in Q} \uq_\mu$
into weight spaces.

For a group homomorphism $\xi: Q\to k^\times$, write
$\ad(\xi)$ for the automorphism of $\uq$ (resp. $\widetilde{\uq}$)
that acts as $\xi(\mu)$ on $\uq_\mu$.
Write $\tilde{H}$ for the group generated by such $\ad(\xi)$.
In particular, we can view $X$ as subgroup of $\tilde{H}$
by having $\mu\in X$ correspond to $\ad(\xi)$ where
$\xi: h\mapsto q^{\langle h,\mu\rangle}$.

For $w\in W$ write $T_w$ for the automorphism $T''_{w,1}\uq\to\uq$ (resp.
$\widetilde{\uq}\to\widetilde{\uq}$) from
\cite[\S37.1.3, \S39.4.7]{Lus10}.
It satisfies $T_w(K_h)=K_{wh}$ for $h\in Y$ (resp. $h\in Y\otimes Q$).

For $\lambda\in X^+$ we write $L(\lambda)$ for the
finite-dimensional $\uq$- (resp. $\widetilde{\uq}$-) module with
highest weight $\lambda$.

\subsection{Quantum Symmetric Pairs}
\begin{notation}
    For any subset $I_\bullet\subset I$ write $w_\bullet$ for the
    longest element of the parabolic subgroup of $W$ generated by
    $I_\bullet$.
\end{notation}

\begin{definition}
    A \emph{compatible admissible pair} is a pair $(I_\bullet, \tau)$
    of a subset $I_\bullet\subset I$ and an involutive diagram
    automorphism $\tau: I\to I$ such that
    \begin{enumerate}
        \item $\tau|_{I_\bullet}=-w_\bullet$;
        \item if $i\in I\setminus I_\bullet$ and
        $\tau(i)=i$, then $\langle\rho_\bullet^\vee,\alpha_i\rangle\in\Z$ (where $\rho_\bullet^\vee$ is the
        half-sum of positive coroots of the root system generated by
        $I_\bullet$); and
        \item there is an involution $\tau$ of $X$ (and hence
        also on $Y$ if we require the pairing to be invariant) 
        such that
        $\tau\alpha_i = \alpha_{\tau(i)}$.
    \end{enumerate}
    Write $I_\circ := I\setminus I_\bullet$ (\enquote{the set of white nodes}) and
    \[
    I_{\mathrm{ns}}:=\set{i\in I_\circ\where\tau(i)=i, \forall j\in I_\bullet\colon{} \langle h_i,\alpha_j\rangle=0}.
    \]
\end{definition}

Let $(I_\bullet,\tau)$ be a compatible admissible pair.
Define
\[
    \mathcal{S}:= \set{i\in I_{\mathrm{ns}}\where
    \forall j\in I_{\mathrm{ns}}:
    i=j\vee \langle h_j,\alpha_i\rangle=-2}
\]
(\enquote{set of non-standard nodes})
and $\Theta := -w_\bullet\circ\tau$.
This is an involution on $X$ and $Y$, as well as $Q,Q^\vee, P,P^\vee$.

For $\mu\in X$ set $\tilde{\mu}:= \frac{\mu-\Theta(\mu)}{2}$.
Let $\Sigma:= \set{\tilde{\alpha}\where\alpha\in R}\setminus\set{0}\subset\frac{1}{2}Q$ be the \emph{restricted root system}.

Write $2L\subset X$ for the set of elements $\mu\in X$ such that
$\langle Y^\Theta, \mu\rangle=0$ and
\[
    \forall \alpha\in\Sigma\colon\quad
    \frac{\alpha\cdot\mu}{\alpha\cdot\alpha} \in\Z.
\]
Note that the quotient map $X\to P$ maps $2L$ to $P(2\Sigma)$, the
weights of $2\Sigma$.

Write $\uq_\bullet$ for the subalgebra of $\uq$ generated by
\[
    (E_i)_{i\in I_\bullet}, (F_i)_{i\in I_\bullet}, (K_{\alpha_i}^{\pm1})_{i\in I_\bullet}.
\]
For $i\in I_\circ$ and two numbers $\bm{c}_i\in k^\times, \bm{s}_i\in k$, define
\[
    B_i := F_i + \bm{c}_i T_{w_\bullet}(E_{\tau(i)})K_{-\alpha_i}
    + \bm{s}_i K_{-\alpha_i}.
\]
\begin{definition}
    Let $\bm{c}\in (k^\times)^{I_\circ}$ and
    $\bm{s}\in k^{I_\circ}$ satisfy the following conditions:
    \begin{enumerate}
        \item If $\alpha_i\cdot\Theta(\alpha_i)=0$, we have
        $\bm{c}_i=\bm{c}_{\tau(i)}$;
        \item $\overline{\bm{c}_i}=\bm{c}_i^{-1}$;
        \item $\bm{c}_i\bm{c}_{\tau(i)}=(-1)^{\langle 2\rho_\bullet^\vee,\alpha_i\rangle} q^{\alpha_i\cdot(\Theta(\alpha_i)-2\rho_\bullet)}$;
        \item if $\bm{s}_j\ne0$, we have $j\in \mathcal{S}$;
        \item $\overline{\bm{s}_i}=\bm{s}_i$.
    \end{enumerate}
    Define $\uqb=\uqb_{\bm{c},\bm{s}}$ to be the subalgebra of
    $\uq$ generated by $\uq_\bullet$, $k[K_h\mid h\in Y^\Theta]$, and
    the $B_i$ with $i\in I_\circ$.

    Similarly, let $\widetilde{\uqb}= \widetilde{\uqb}_{\bm{c},\bm{s}}$
    be the subalgebra of $\widetilde{\uq}$ generated by $\uq_\bullet$,
    $k[K_h\mid h\in (Y\otimes\Q)^\Theta]$, and the $B_i$ for
    $i\in I_\circ$.
    We call $(\uq,\uqb)$ the \emph{quantum symmetric pair} (QSP) for the
    Cartan data $(X,Y,A)$, the compatible admissible pair $(I_\bullet,\tau)$, and the parameters $\bm{c},\bm{s}$.
    Similarly, $(\widetilde{\uq},\widetilde{\uqb})$ is the
    completed QSP for the same data.
\end{definition}

\begin{remark}
    Conditions (i) and (iv) on the data are the standard conditions on
    parameters that can also be found in \cite{Kol14}.
    Conditions (ii), (iii), (v) have been added so that the results of
    \cite{Wat23} hold; in particular that simple $\uq$-modules are
    semisimple over $\uqb$, and the existence of a
    bar involution on $\uqb$.

    Of these three extra conditions, (iii) does not represent a large
    restriction as this can always be effected by an appropriate
    action of $\tilde{H}$.
    (ii) and (v), however, are proper restrictions.
    Concerning the results of this paper, they can be eliminated
    using a Zariski argument.
\end{remark}

\begin{lemma}
    The elements $K_{-\alpha_i}E_i$ and $F_i$ generate finite-dimensional
    $\ad_r(\uq_\bullet)$-submodules of $S(\uq^+)$ and $\uq^-$, 
    respectively.
\end{lemma}
\begin{proof}
    Apply the antipode $S$ to \cite[Lemma~3.5]{Kol14}.
\end{proof}

Write $\mathcal{N}^+$ and $\mathcal{N}^-$ for the subalgebras of
$S(\uq^+)$ and $\uq^-$
generated by $\adr(\uq_\bullet)(K_i^{-1}E_i)$ and 
$\adr(\uq_\bullet)(F_i)$ ($i\in I_\circ$), respectively.

For $i\in I_\circ$, write $I[i]$ for the rank-1 subdiagram for $i$;
i.e. the set of all nodes connected to $i$ or $\tau(i)$ in
$\set{i,\tau(i)}\cup I_\bullet$ in the Satake diagram associated with
$(I_\bullet,\tau)$.

\begin{remark}
    Conditions (i) and (iv) on the parameters show that
    $\bm{c}_i=\bm{c}_{\tau(i)}$ in case $I[i]$ is of type
    $\operatorname{diag}(\mathsf{A}_1)$, and that
    $\bm{s}_i\ne0$ is only allowed in case $I[i]$ is of type
    $\mathsf{AI}_1$ and if $\tilde{\alpha}_i$ is the long simple root
    in a restricted root system $\Sigma$ of type $\mathsf{C}_r$.
\end{remark}

\subsection{Spherical Functions}
For the rest of the paper we will fix two
QSPs $(\uq,\uqb)$ and $(\uq,\uqb')$ for the same
compatible admissible pair and potentially different parameters
($\uqb=\uqb_{\bm{c},\bm{s}}$ and $\uqb'=\uqb_{\bm{d},\bm{t}}$).

\begin{definition}
    Let $\chi:\uqb\to k$ and $\chi':\uqb'\to k$ be characters.
    A function $f\in\mathcal{A}$ is called a \emph{character spherical
    function} for $\chi',\chi$ if
    \[
        \forall b'\in\uqb',x\in\uq,b\in\uq\colon\quad
        f(b'xb) = \chi'(b') \chi(b) f(x).
    \]
    In case $\chi=\chi'=\epsilon$, $f$ is called a \emph{zonal spherical function}.
    Write $E^{\epsilon}=E^{\epsilon,\epsilon}_{\uqb',\uqb}$ for the
    vector space of zonal spherical functions, and
    $E^{\chi',\chi}=E^{\chi',\chi}_{\uqb',\uqb}$ for the vector
    space of $(\chi',\chi)$-character spherical functions.
\end{definition}

\begin{definition}
    Let $M$ be a simple $\uq$-module, and write $\mathcal{A}(M)$ for its matrix elements.
    A character spherical function $f\in E^{\chi',\chi}\cap \mathcal{A}(M)$ is called \emph{elementary}.
    Write $E^{\chi',\chi}_{\uqb',\uqb}(\mu) = E^{\chi',\chi}(\mu)$ for the vector space of elementary character spherical functions for $M=L(\mu)$.

    Write $X^+(\chi,\chi')$ for the set of all $\mu\in X^+$
    such that $E^{\chi',\chi}(\mu)\ne0$.
\end{definition}

\begin{proposition}
    Let $X^+(\chi,\chi')\ne\emptyset$.
    Then there exists a unique element $b\in X^+$ such that
    for all $\mu\in X^+$ the following are equivalent:
    \begin{enumerate}
        \item $\dim(E^{\chi',\chi}(\mu))=1$;
        \item $\mu-b\in 2L$.
    \end{enumerate}
\end{proposition}
\begin{proof}
    By \cite[Lemma~2.7]{Mee25} there are $b_1, b_2\in X^+$
    such that $L(\mu)$ contains a $\uqb'$-submodule
    isomorphic to $\chi'$ iff $\mu-b_1\in 2L^+$,
    and such that $L(\mu)$ contains a $\uqb$-submodule
    isomorphic to $\chi$ iff $\mu-b_2\in 2L^+$.
    
    By \cite[Lemma~4.2]{Mee25sym}, $(\uq,\uqb,\chi)$ and
    $(\uq,\uqb',\chi')$ are multiplicity-free triples,
    i.e. $\dim(E^{\chi',\chi}(\mu))\le1$ in any case.
    Consequently, by \cite[Theorem~4.11]{Mee25sym},
    (i) is equivalent to $L(\mu)$ containing a
    $\uqb'$-sub\-mod\-ule isomorphic to $\chi'$, and
    a $\uqb$-submodule isomorphic to $\chi$,
    hence to $\mu\in (b_1+2L^+)\cap (b_2+2L^+)$.
    This shows that $X^+(\chi,\chi')=(b_1+2L^+)\cap (b_2+2L^+)$, which is again of the shape
    $b+2L^+$ since it is nonempty.
\end{proof}

\begin{proposition}
    $E^{\epsilon}$ is a ring,
    and $E^{\chi',\chi}$ is a module over $E^{\epsilon}$.
    Let $X^+(\chi,\chi')=b+2L^+$ and $0\ne\phi\in E^{\chi',\chi}(b)$,
    then $E^{\chi',\chi}$ is a free $E^\epsilon$ module
    with basis $\phi$.
\end{proposition}
\begin{proof}
    Follows from \cite[Proposition~2.12]{qmsf} and a proof analogous to
    \cite[Theorem~4.3]{Mee25}.
\end{proof}

\begin{definition}\label{def-res}
    Define the \emph{restriction function} to be the following map
    \[
        \Res'\colon\:\mathcal{A}\to k[X],\qquad
        c^M_{f,v} \mapsto f(v)e^\mu
    \]
    whenever $f\in M^*$ and $v\in M_\mu$ is a vector of
    weight $\mu\in X$, and extending linearly.

    Let moreover $\Res\colon\:\mathcal{A}\to k[X]$ be
    the \emph{shifted restriction function}, defined
    by
    \[
        \Res(\phi):= \Res'(\phi)(\cdot-\rho),
    \]
    i.e. by composing $\Res$ with the ring
    automorphism of $k[X]$ mapping
    $e^\mu\mapsto q^{-\rho\cdot\mu} e^\mu$.
\end{definition}

\begin{remark}
    In the language of the modified form $\dot{\uq}$ (cf.
    \cite[\S23]{Lus10}), the definition of $\Res'$ can be simplified
    as follows:
    \[
        \Res'(f) = \sum_{\mu\in X} f(1_\mu)e^\mu.
    \]
\end{remark}

\begin{lemma}
    Both $\Res$ and $\Res'$ are ring homomorphisms.
    Moreover, $\Res'$ restricts to a ring isomorphism
    $E^\epsilon\cong k[2L]^{W_\Sigma}$.
\end{lemma}
\begin{proof}
    It suffices to show that $\Res'$ is a ring homomorphism.
    We have
    \begin{align*}
        \Res'(c^M_{f,v} c^{M'}_{f',v'})
        &= \Res'(c^{M\otimes M'}_{f'\otimes f, v\otimes v'})
        = (f'\otimes f)(v\otimes v')
        e^{\mu+\mu'}\\
        &= f(v) f'(v') e^{\mu+\mu'}
        = \Res'(c^M_{f,v})\Res'(c^{M'}_{f',v'}).
    \end{align*}
    The last claim follows from \cite[Lemma~3.58]{qmsf}
    and \cite[Corollary~5.4]{Le03}.
\end{proof}

\section{Centraliser of \texorpdfstring{$\widetilde{\uqb}$}{B}}\label{sec-centraliser}
Our goal is to identify the shifted restrictions of ZSF
(and character spherical functions) with Macdonald--Koornwinder 
polynomials by identifying them as eigenfunctions of sufficiently
many $q$-difference operators that happen to also have the
Koornwinder--Macdonald polynomials as eigenfunctions.

For this we take a closer look at the centraliser $Z_{\widetilde{\uq}}(\widetilde{\uqb})$.
Write $Y\otimes_\Z\Q\cong T\oplus A$, where $T$ consists of the
$\Theta$-invariants and $A$ of the anti-invariants.

\begin{lemma}\label{lem-existence-radial-part}
    Let $x\in Z_{\widetilde{\uq}}(\uqb)$ and let $h\in Y\otimes_\Z\Q$,
    then there is a unique $q$-difference operator
    $\Rad'(x)\in k(Q)[A]$, say
    \[
        \Rad'(x) = \sum_{\mu\in A} a_\mu T_\mu
    \]
    (we write $T_he^\mu := q^{\langle h,\mu\rangle}e^\mu$, and by abuse of 
    notation, $T_\lambda e^\mu= q^{\lambda\cdot\mu}e^\mu$)
    such that for a Zariski dense subset of $h\in Y\otimes_\Z\Q$ we have
    \[
        K_h x - \sum_{\mu\in Y} a_\mu(h) K_{h+\mu}
        \in \widetilde{\uqb}'_+\widetilde{\uq}
        + \widetilde{\uq}\widetilde{\uqb}_+.
    \]
\end{lemma}
\begin{proof}
    By a straightforward generalisation of \cite[Chapter~4]{qmsf}, we have
    \[
        K_h x \in \widetilde{\uqb}'\widetilde{\uq}^0\widetilde{\uqb}
    \]
    for a dense set of $h$.
    Decomposing $\widetilde{\uqb} = k \oplus\widetilde{\uqb}_+$ and $\widetilde{\uqb}' = k\oplus\widetilde{\uqb}'_+$, as well
    as $\widetilde{U}^0=k[T]\otimes k[A] = k[A] \oplus k[A]k[T]_+$ we obtain
    \[
        K_h x \in k[A]\oplus(\widetilde{\uqb}'_+\widetilde{\uq}
        + \widetilde{\uq}\widetilde{\uqb}_+).
    \]
    From the proof of \cite[Chapter~4]{qmsf} we furthermore know that the coefficients are rational functions in $h$, rational
    functions contained in $k(Q)$.
\end{proof}

\begin{definition}\label{def-radial-part}
    In the setting of Lemma~\ref{lem-existence-radial-part}, 
    $\Rad'(x)$ is called the \emph{radial part} of $x$.

    Moreover, we define the \emph{shifted radial part}
    $\Rad(x)\in k(Q)[A]$ as follows:
    \[
        \Rad(x)(h) = \Rad'(x)(h-\rho),
    \]
    i.e. if $\Rad'(x) = \sum_{\mu\in A} a_\mu T_\mu$, we have
    \[
        \Rad(x) = \sum_{\mu\in A} a_{\mu}(\cdot -\rho) T_\mu.
    \]
\end{definition}

\begin{lemma}\label{lem-action-rad}
    Let $x\in Z_{\widetilde{\uq}}(\widetilde{\uqb})$ and
    $\phi\in E^{\epsilon}_{\uqb',\uqb}$.
    Then
    \[
        \Res(x\triangleright \phi)
        = \Rad(x) \Res(\phi),
    \]
    where $x\triangleright\phi\in\mathcal{A}$ is the function mapping
    $y\mapsto \phi(yx)$.
    Moreover, $\Rad$ is invariant under $W_\Sigma$.
\end{lemma}
\begin{proof}
    Write $\Rad'(x) = \sum_{\mu\in A} a_\mu K_\mu$, then
    \begin{align*}
        \Res(x\triangleright\phi)(h)
        &= \phi(K_{h-\rho} x)\\
        &= \sum_{\mu\in A} a_\mu(h-\rho) \phi(K_{h-\rho+\mu})\\
        &= \sum_{\mu\in A} a_\mu(h-\rho) \Res(\phi)(h+\mu)\\
        &= \Rad(x)\Res(\phi)(h).\qedhere
    \end{align*}
\end{proof}

\begin{definition}\label{def-hc}
    Define $\gamma\colon \widetilde{\uq}\to k[A]$, the
    \emph{Harish-Chandra map} $\gamma$, as follows:
    \[
        \widetilde{\uq} \cong \mathcal{N}^-\otimes k[A]\otimes\uqb
        \xrightarrow{\epsilon\otimes\id\otimes\epsilon} k[A].
    \]
\end{definition}

\begin{lemma}\label{lem-gamma-surjective-morphism}
    $\gamma$'s restriction to $Z_{\widetilde{\uq}}(\widetilde{\uqb})$
    is a ring homomorphism mapping surjectively to (suitably $\rho$-shifted)
    $W_\Sigma$-invariants of $k[2\tilde{L}]$ (where we obtain 
    $\tilde{L}$ from the definition of $L$ by replacing $X$ with
    $X\otimes\Q$).
    Moreover, we have
    \[
        \forall \lambda\in 2L, \phi\in E^\epsilon(\lambda), x\in Z_{\widetilde{\uq}}(\widetilde{\uqb}): \quad x\triangleright\phi = \gamma(x)(\lambda)\phi.
    \]
\end{lemma}
\begin{proof}
    Note that our definition of $\gamma$ corresponds to
    $S\circ\mathcal{P}_{S(\uqb)}\circ S$ from \cite{Le08}.
    
    By loc.cit., Theorem~4.4, $\gamma$ is a ring homomorphism.
    By a proof analogous to loc.cit., Theorem~4.6, $\gamma$ 
    maps to $k[2\tilde{L}]$.
    By loc.cit., Theorem~A, $\gamma$ is surjective onto the
    set of dotted $W_\Sigma$-invariants in $k[2\tilde{L}]$.
    Lastly, by a proof analogous to loc.cit., Lemma~5.3,
    $\gamma$ describes the eigenvalues of elementary MSF under
    the action of $Z_{\widetilde{\uq}}({\widetilde{\uqb}})$.
\end{proof}

\begin{lemma}\label{lem-can-choose-good-operator}
    For $\mu\in 2L^+$, there exists
    $x\in Z_{\widetilde{\uq}}(\widetilde{\uqb})$ such that
    $\gamma(x)=K_\mu + \text{l.o.t.}$ and such that
    $\Rad(x)$ is supported on elements in the saturated hull of
    $T_\mu$.
\end{lemma}
\begin{proof}
    Follows from the surjectivity assertion in Lemma~\ref{lem-gamma-surjective-morphism}
    and \cite[Lemma~5.8]{Le08}.
\end{proof}

If $x$ and $\mu$ are as in Lemma~\ref{lem-can-choose-good-operator}, we can write
$\Rad(x) = aK_\mu+\lot$ for a rational function $a\in k(Q)$.
We will spend the following four sections to determine $a$ for the
following reason:
as showed in Lemma~\ref{lem-p-determines-DO} (in that context
we have $pT_\omega p^{-1}=a T_\omega$), $\mu$ is pseudominuscule for $2\Sigma$, $a$ (and $\gamma(x)$)
already determines $\Rad(x)$.

In particular, as will be shown in Theorem~\ref{thm-initial-term-given-by-p}, there is a power
series $p$ that is independent of $x$ (or $\mu$) such that
$\Rad(x) = p T_\mu p^{-1} + \lot$.

\section{Graded Rings and Modules}\label{sec-graded}
In this section, we will show $p$'s existence.
To achieve this, it is useful to rigorously be able to neglect
lower-order terms, so we introduce a filtration on $\widetilde{\uq}$
and consider its associated graded ring.
In particular, this construction is necessary to make Lemma~\ref{lem-definition-p'} possible and to show that $p'$ (and hence also $p$)
is independent of $\mu$.

\begin{lemma}
    Let $\mu\in A$ and define
    \[
        \widetilde{\uq}_\mu = \uq^- K_\mu k[T] S(\uq^+)
    \]
    and
    \[
        \mathcal{F}_\lambda(\widetilde{\uq}) := \bigoplus_{\lambda-\mu\in Q(\Sigma)^+} \widetilde{\uq}_\mu.
    \]
    Then $(\mathcal{F}_\lambda(\widetilde{\uq}))_{\lambda\in A}$ is
    a filtration of rings.
\end{lemma}
\begin{proof}
    To show this, we note that the vector space grading assigns to the
    generators $F_i, E_i, K_\mu$ the degrees $0, \tilde{\alpha}_i, \tilde{\mu}$, respectively.
    Furthermore, the relations $K_\lambda K_\mu = K_{\lambda+\mu}$
    \[
        K_\mu E_i K_{-\mu} = q^{\langle\mu,\alpha_i\rangle} E_i,\qquad
        K_\mu F_i K_{-\mu} = q^{-\langle\mu,\alpha_i\rangle} F_i,\qquad
    \]
    and the Serre relations are homogeneous, and the relation
    \[
        \comm{E_i}{F_j} = \delta_{ij} \frac{K_i-K_{i}^{-1}}{q_i-q_i^{-1}}
    \]
    is homogeneous modulo terms of order lower than $\tilde{\alpha}_i$.
    Consequently, $\mathcal{F}_\lambda\mathcal{F}_\mu\subset\mathcal{F}_{\lambda+\mu}$.
\end{proof}

\begin{remark}
    Our filtration is related to the one defined in
    \cite[Chapter~5]{Le04}:
    Let $\alpha\in X$ be written as
    \[
        \alpha = \sum_{i\in I_\circ} n_i \alpha_i
    \]
    and let $n=\sum_{i\in I_\circ} n_i$.
    The subspace $S(\mathcal{F}_{h_\alpha})$ then consists of
    elements of degree $n$ or less (in Letzter's language).
\end{remark}

\begin{lemma}\label{lem-adr-invariance-of-filtration}
    For every $\mu\in A$, $\mathcal{F}_\mu(\widetilde{\uq})$ is $\adr(\uq)$-invariant.
\end{lemma}
\begin{proof}
    Let $x\in \uq_\mu$, say
    \[
        x = \sum_i F_{I_i} f_i(K) K_\mu S(E_{J_i})
    \]
    for multi-indices $I_i, J_i$, and polynomials
    $f_i\in k[T]$.
    Then
    \[
        \adr(K_h)(x) = \sum_i q^{\langle h, \wt(J_i)-\wt(I_i)\rangle}
        F_{I_i} f_i(K) K_\mu S(E_{J_i})\in\uq_\mu.
    \]
    For $j\in I$, we furthermore have
    \begin{align*}
        \adr(E_i)(x) &= S(E_i)x
        - \adr(K_{\alpha_i})(x) S(E_i)\in\mathcal{F}_\mu(\widetilde{\uq}),\\
        \adr(F_i)(x) &= x F_i - F_i \adr(K_{\alpha_i})(x)\in\mathcal{F}_\mu(\widetilde{\uq}).
    \end{align*}
\end{proof}

\begin{lemma}
    We have $\uq_\bullet, k[T],\uq^- \subset \widetilde{\uq}_0$.
    Furthermore, $B_i \in \widetilde{\uq}_0 \oplus \widetilde{\uq}_{-\tilde{\alpha}_i}$ where the $-\tilde{\alpha}_i$ component is
    proportional to the non-standard parameter $\bm{s}_i$.
\end{lemma}
\begin{proof}
    For $\uq_\bullet$ and $k[T]$, and $\uq^-$ note that
    \[
        E_i\in\widetilde{\uq}_{\tilde{\alpha}_i},\quad 
        F_i\in\widetilde{\uq}_0,\qquad
        K_\mu \in \widetilde{\uq}_{\tilde{\mu}},
    \]
    and that $\tilde{\alpha}_i=0$ for $i\in I_\bullet$, and that
    $\tilde{\mu}=0$ for $\mu\in T$.

    Lastly, for $i\in I_\circ$ we have
    \[
        B_i = F_i + c_i T_{w_\bullet}(E_{\tau(i)})K_i^{-1} + s_i K_i^{-1},
    \]
    where $F_i\in\widetilde{\uq}_0$, $T_{w_\bullet}(E_{\tau(i)})\in U^+_{w_\bullet\alpha_{\tau(i)}}$, and hence has degree
    $\widetilde{w_\bullet\alpha_{\tau(i)}} = \tilde{\alpha}_i$.
    Note however that $K_i^{-1}$ has degree $-\tilde{\alpha}_i$, so that
    the middle term of $B_i$ also has degree 0.
    The last term of $B_i$ has degree $-\tilde{\alpha}_i$.
\end{proof}

\begin{lemma}\label{lem-B-embedded}
    The map $\widetilde{\uqb}\to\gr(\widetilde{\uqb})$
    mapping 
    \[
        E_i, F_i, K_\mu, B_j\mapsto \gr(E_i), \gr(F_i), \gr(K_\mu),
        \gr(B_j)
    \]
    for $i\in I_\bullet, j\in I_\circ, \mu\in T$ is a ring isomorphism.
\end{lemma}
\begin{proof}
    In case $\bm{s}=0$, all generators of $\widetilde{\uqb}$ have
    degree zero, so the filtration restricts to the trivial filtration
    and the statement is obviously true.

    Furthermore, by \cite[Section~7]{Kol14}, as an algebra, 
    $\uqb$ and hence also $\widetilde{\uqb}$ does not depend on
    $\bm{s}$, so that the claim also holds for $\bm{s}\ne0$.
\end{proof}

\begin{lemma}
    The grading on $\mathcal{N}^\pm$ is trivial, so the initial
    maps $\mathcal{N}^\pm\to\gr(\mathcal{N}^\pm)$ are ring isomorphisms.
\end{lemma}
\begin{proof}
    This follows since $\mathcal{N}^+\subset S(\uq^+)$ and $\mathcal{N}^-\subset \uq^-$.
\end{proof}

Now we define the graded equivalent of Verma modules.
The goal is to realise $p$ as the shifted restriction of a
\enquote{universal} zonal spherical function
\begin{definition}
    Let $\lambda\in (X\otimes_\Z\Q)^{-\Theta}$ and define the left $\gr(k[A\oplus T])$-module
    $kv_\lambda$ by having $K_h v_\lambda = q^{\langle h,\lambda\rangle}$, and let furthermore
    $\gr(\uq_\bullet)$ and $\gr(S(\uq^+))$ act trivially on $v_\lambda$.

    Define $\overline{M}(\lambda):= \gr(\widetilde{\uq})\otimes_{\gr(\uq_\bullet k[A\oplus T] S(\uq^+))} kv_\lambda$.

    Similarly, we can have $kv_\lambda$ be a right
    $\gr(\uq_\bullet k[A\otimes T]\uq^-)$-module and define
    \[
        \overline{M}(\lambda)^{\text{r}} :=
        kv_\lambda\otimes_{\gr(\uq_\bullet k[A\oplus T]\uq^-)} \gr(\widetilde{\uq}).
    \]
\end{definition}

\begin{lemma}\label{lem-mlambda-N}
    We have $\overline{M}(\lambda) = \gr(\mathcal{N}^-)\otimes kv_\lambda$ and $\overline{M}(\lambda)^{\mathrm{r}} = kv_\lambda\otimes \gr(\mathcal{N}^+)$.
\end{lemma}
\begin{proof}
    This is a consequence of the quantum Iwasawa decompositions
    \[
        \widetilde{\uq}\cong \mathcal{N}^\pm\otimes k[A]\otimes\widetilde{\uqb}.\qedhere
    \]
\end{proof}

\begin{lemma}\label{lem-action-on-uqminus}
    $S(\uq^+)$ acts as follows on $\uq^-$: for $x\in S(\uq^+)$ and
    $y\in\uq^-$ we have
    \[
        \adr(S^{-1}(x))(y) \in \bigoplus_{\alpha,\beta\in Q(R)^+}\uq^-
        K_{-2\alpha}S(\uq^+)_\beta.
    \]
    Let $\pi_{0}$ be the projection to the $\alpha=\beta=0$ summand on the
    right-hand side, then define
    \[
        x\ast y := \pi_0(\adr(S^{-1}(x)))(y).
    \]
    This defines a representation of $S(\uq^+)$ on $\uq^-$.
\end{lemma}
\begin{proof}
    Write $x:= F_J K_{-2h_\alpha}S(E_K)$, then
    \begin{align*}
        -\adr(E_i)(x) &= K_{-\alpha_i} E_i x - K_{-\alpha_i}xE_i\\
        &= K_{-\alpha_i} E_i F_J K_{-2\alpha} S(E_K)
        - K_{-\alpha_i} F_J K_{-2\alpha}  S(E_K)E_i\\
        &= K_{-\alpha_i} \comm{E_i}{F_J} K_{-2\alpha} S(E_K)\\
        &\quad+ q^{\alpha_i\cdot (\wt(J)+2\alpha)} F_J K_{-2\alpha} K_{-\alpha_i}E_iS(E_K)\\
        &\quad- q^{\alpha_i\cdot (\wt(J)-\wt(K))}
        F_J K_{-2\alpha} S(E_K)K_{-\alpha_i}E_i\\
        &= K_{-\alpha_i} \comm{E_i}{F_J} K_{-2\alpha} S(E_K)\\
        &\quad- q^{\alpha_i\cdot(\wt(J)+2\alpha)} F_J K_{-2\alpha}S(E_KE_i)\\
        &\quad+ q^{\alpha_i\cdot (\wt(J)-\wt(K))}
        F_J K_{-2\alpha} S(E_iE_K).
    \end{align*}
    Here, $\comm{E_i}{F_J}\in \uq^-(K_{\alpha_i}\oplus K_{-\alpha_i})$, which shows the claim.
    Since all terms not projected onto by $\pi_0$ only ever increase
    the weight of $K$ or decrease the argument of $K$, the map
    thus defined is a representation.
\end{proof}

\begin{lemma}\label{lem-isomorphic-suplus-modules}
    Let $\lambda\in (X\otimes_\Z\Q)^{-\Theta}$, then
    for $x\in\mathcal{N}^-$ we have
    \[
        y x\otimes v_\lambda = (y\ast x)\otimes v_\lambda
    \]
    for $y\in S(\uq^+)$.
    In other words, as $S(\uq^+)$-modules we have $\overline{M}(\lambda)\cong\mathcal{N}^-$.

    In particular, the $S(\uq^+)$-module $\overline{M}(\lambda)$
    does not depend on $\lambda$.
\end{lemma}
\begin{proof}
    Let $i\in I$ then
    \begin{align*}
        K_i^{-1}E_i x\otimes v_\lambda
        &= (K_i^{-1}E_i x - K_i^{-1}xE_i)\otimes v_\lambda
        = -\adr(E_i)(x)\otimes v_\lambda\\
        &= \adr(S^{-1}(K_i^{-1}E_i))(x)\otimes v_\lambda.
    \end{align*}
    Note that the terms projected away in the definition of $\ast$
    precisely vanish in $\gr(\widetilde{\uq})$ (namely the terms with
    $\alpha>0$) or annihilate $v_\lambda$ (namely the summands with $\beta>0$).
    We conclude
    \[
        K_i^{-1}E_i x\otimes v_\lambda
        = (K_i^{-1}E_i)\ast x\otimes v_\lambda.\qedhere
    \]
\end{proof}

\begin{lemma}\label{lem-simple-cyclic}
    $\overline{M}(\lambda)$ is irreducible over $\gr(\widetilde{\uq})$,
    and cyclic over $\gr(\widetilde{\uqb})$ (independently of the
    parameters).

    In particular, as $\widetilde{\uqb}\cong\gr(\widetilde{\uqb})$-modules, we have $\overline{M}(\lambda)\cong M(\lambda)$ from
    \cite{Wat23}.
\end{lemma}
\begin{proof}
    We need to show that every vector $x\otimes v_\lambda$ for
    $x\in\mathcal{N}^-$ is cyclic, which is equivalent to showing that
    only $1\otimes v_\lambda$ is annihilated by $S(\uq^+)_+$,
    but that follows from the definition of the action $\ast$.

    Furthermore, note that $\widetilde{\uq}\cong \widetilde{\uqb}\otimes k[A]\otimes\mathcal{N}^+$.
    This translates to the graded version as
    \[
        \gr(\widetilde{\uq})= \gr(\widetilde{\uqb}) \gr(k[A])\mathcal{N}^+,
    \]
    consequently
    \[
        \gr(\widetilde{\uq})\otimes v_\lambda =
        \gr(\widetilde{\uqb})\otimes v_\lambda.\qedhere
    \]
\end{proof}

\begin{lemma}
    Let $F(\overline{M}(\lambda)^*)$ be the direct sum of
    $\overline{M}(\lambda)^*$'s weight spaces.
    Then $F(\overline{M}(\lambda)^*\cong \overline{M}(\lambda)^{\mathrm{r}}$ as $\gr(\widetilde{\uq})$-modules.
\end{lemma}
\begin{proof}
    Let $v_\lambda^*\in F(\overline{M}(\lambda)^*)$ be the linear
    form mapping $v_\lambda\to 1$ and every other weight space to 0.
    Then $v_\lambda^*$ is a lowest-weight vector of weight $\lambda$:
    let $i\in I$ and $v\in\overline{M}(\lambda)$. Then $v$ has weight
    $\le\lambda$.
    In particular, $F_iv$ has weight $<\lambda$, so
    $v_\lambda^*(F_iv)=0$, whence $v_\lambda^* F_i=0$.
    Moreover, since $v_\lambda$ generates the trivial
    $\gr(\uq_\bullet)$-module, the same is true for $v_\lambda^*$.

    Consequently, there is a canonical $\gr(\widetilde{\uq})$-homomorphism
    $\overline{M}(\lambda)^{\mathrm{r}}\to F(\overline{M}(\lambda)^*)$
    mapping $v_\lambda$ to $v_\lambda^*$.
    For dimension reasons, this homomorphism is an isomorphism.
\end{proof}

\begin{corollary}
    The claims from Lemma~\ref{lem-simple-cyclic} are also true
    for $\overline{M}(\lambda)^{\mathrm{r}}$.
\end{corollary}

\begin{lemma}\label{lem-verma-adjunction}
    Let $V$ be a finite-dimensional simple right $\widetilde{\uqb'}$-module,
    and let $W$ be a finite-dimensional simple left $\widetilde{\uqb}$-module.
    There are vector space isomorphisms
    \begin{align*}
        \Hom_{\widetilde{\uqb'}}(V, \overline{M}(\lambda)^*)
        &\cong \Hom_{\uq_\bullet}(kv_\lambda, V^*)\\
        \Hom_{\widetilde{\uqb}}(W, \overline{M}(\lambda)^{\mathrm{r}*})&\cong \Hom_{\uq_\bullet}(kv_\lambda, W^*).
    \end{align*}
\end{lemma}
\begin{proof}
    Let $f\in\Hom_{\widetilde{\uqb'}}(V, \overline{M}(\lambda)^*)$.
    Define
    \[
        \phi(f)\colon{} kv_\lambda\to V^*,\qquad
        v_\lambda\mapsto (v\mapsto f(v)(v_\lambda)).
    \]
    Furthermore, for $x\in\uq_\bullet$ we have
    \[
        \phi(f)(xv_\lambda)(v)
        = f(v)(xv_\lambda)
        = f(vx)(v_\lambda)
        = (x\phi(f)(v_\lambda))(v)
    \]
    because $f$ intertwines $\uq_\bullet\subset\widetilde{\uqb'}$.
    We conclude $\phi(f)(xv\lambda) = x\phi(f)(v_\lambda)$.
    Thus the map $\phi$ is a well-defined linear map.
    Furthermore, by Lemma~\ref{lem-simple-cyclic} $f(v)\in\overline{M}(\lambda)^*$ is
    fully determined by its action on $v_\lambda$, i.e.
    by $\phi(f)(v_\lambda)(v)$.
    Consequently, $\phi$ is injective.
    For surjectivity, let $g\in\Hom_{\uq_\bullet}(kv_\lambda, V^*)$
    and define
    \[
        f(v)(b v_\lambda) :=
        g(v_\lambda)(vb).
    \]
    This is well-defined for the following reason:
    as $\uqb'$-modules, $\overline{M}(\lambda)$ is isomorphic
    to $M(\lambda)$ from \cite{Wat23} (for nonstandard parameter
    $\bm{t}=0$).
    By \cite[Proposition~3.3.3]{Wat23}, the only relations satisfied
    by $v_\lambda$ are $K_h v_\lambda= v_\lambda$ with
    $h\in T$, and $E_i v_\lambda=0$ for $i\in I_\bullet$.
    Both are satisfied since $g$ is a $\uq_\bullet$-intertwiner,
    and since $\lambda\perp T$.

    A similar statement follows for the second equality.
\end{proof}

\begin{corollary}\label{cor-unique-invariant-element}
    Write $\widehat{M}(\lambda)$ and $\widehat{M}(\lambda)^{\mathrm{r}}$
    for the direct products of the weight spaces of
    $\overline{M}(\lambda)$ and $\overline{M}(\lambda)^{\mathrm{r}}$,
    so that $\widehat{M}(\lambda)\cong (\overline{M}(\lambda)^{\mathrm{r}})^*$ and
    $\widehat{M}(\lambda)^{\mathrm{r}}\cong(\overline{M}(\lambda))^*$.
    
    Then up to scalars there is a unique $\widetilde{\uqb}$-invariant element
    of $\widehat{M}(\lambda)$ and a unique $\widetilde{\uqb'}$-invariant element of $\widehat{M}(\lambda)^{\mathrm{r}}$.
\end{corollary}
\begin{proof}
    Consider $k$ as $\widetilde{\uqb}$ and $\widetilde{\uqb}'$-module via
    $\epsilon$,
    then by Lemma~\ref{lem-verma-adjunction}, the vector space
    $\overline{M}(\lambda)^{*,\widetilde{\uqb'}}\cong
    \widehat{M}(\lambda)^{\mathrm{r},\widetilde{\uqb'}}$ is
    one dimensional,
    as is the vector space
    \[
        \overline{M}(\lambda)^{\mathrm{r},*,\widetilde{\uqb}}
        \cong
        \widehat{M}(\lambda)^{\widetilde{\uqb}}.\qedhere
    \]
\end{proof}

\begin{definition}
    Write $\widehat{\mathcal{N}}^\pm$ for the direct product of
    $\mathcal{N}^\pm$'s weight spaces.
\end{definition}

\begin{corollary}
    There exist unique elements $b^*\in\widehat{\mathcal{N}}^+$ and
    $b\in\widehat{\mathcal{N}}^-$ such that
    $b^*_0 = b_0 = 1$ and such that
    \[
        b\otimes v_\lambda \in \widehat{M}(\lambda)^{\widetilde{\uqb}},\qquad
        v_\lambda \otimes b^* \in \widehat{M}(\lambda)^{\mathrm{r},\widetilde{\uqb'}}
    \]
    for all $\lambda\in (X\otimes_\Z\Q)^{-\Theta}$.
\end{corollary}
\begin{proof}
    Follows from Corollary~\ref{cor-unique-invariant-element},
    Lemma~\ref{lem-mlambda-N}, and Lemma~\ref{lem-simple-cyclic}, which
    implies that $\overline{M}(\lambda)$ and its completion as a $\widetilde{\uqb}$-module (and $\overline{M}(\lambda)^{\mathrm{r}}$ and its completion as $\widetilde{\uqb'}$-module) is independent
    of $\lambda$.

    To see that $b_0^*,b_0\ne0$, we refer to \cite[Lemma~3.3]{Le03}.
\end{proof}

\begin{definition}
    Define $\Upsilon\colon\:\overline{M}(\lambda)^{\mathrm{r}}\otimes
    \overline{M}(\lambda)\to k[X]$ by
    \[
        f\otimes v\mapsto \sum_{\mu\in X} f(v_\mu) e^\mu
    \]
    where $v=\sum_{\mu\in X}v_\mu$ is the decomposition into weight
    spaces, and where we identify $\overline{M}(\lambda)^{\mathrm{r}}$
    with a subset of $\overline{M}(\lambda)^*$.
    Analogously, we define
    \[
        \widehat{\Upsilon}: \widehat{M}(\lambda)^{\mathrm{r}}\otimes
        \widehat{M}(\lambda)\to k[[X]].
    \]
\end{definition}

\begin{remark}
    Note that $\Upsilon(f\otimes v)(h) = f(K_hv) = c^{\overline{M}(\lambda)}_{f,v}(K_h)$,
    so that $\Upsilon$ is the associated graded version of
    \[
        f\otimes v\mapsto \Res'(c^M_{f,v})
    \]
    (for a $\uq$-module $M$).
    Note that neither $f\otimes v\mapsto c^M_{f,v}$ nor
    $\Res'$ have suitable equivalents for the completions of $\overline{M}(\lambda)^{\mathrm{r}},\overline{M}(\lambda)$,
    only their composition.
    
    Note furthermore that elements in the image of $\widehat{\Upsilon}$
    are supported in the weights of $\overline{M}(\lambda)$, which are $\lambda + Q^-$.
    In particular, $\widehat{\Upsilon}(f\otimes v) = e^\lambda p$,
    where $p$ is a power series in $e^{-\alpha_i}$ for $i\in I$.
\end{remark}

We now expand on this last statement in particular and show that such
a power series $p$ (now $p'$ for consistency reasons) can be chosen
independently of $\lambda$, and as a power series in $e^{-\tilde{\alpha}_i}$ for $i\in I_\circ$.

\begin{lemma}\label{lem-definition-p'}
    There is $p'=1+\lot \in k[[e^{-\tilde{\alpha}_i}, i\in I_\circ]]$
    such that
    \[
        \widehat{\Upsilon}(v_\lambda b^*\otimes bv_\lambda)
        = e^\lambda p'
    \]
    for all $\lambda\in k[X\otimes_\Z\Q]^{-\Theta}$.
\end{lemma}
\begin{proof}
    Write
    \[
        P: \gr(\widetilde{\uq})=\gr(k[T\oplus A])\oplus (\gr(\widetilde{\uq}S(\uq^+)_+)+\gr(\uq^-_+\widetilde{\uq}))
    \to \gr(k[T\oplus A])
    \]
    for the projection.
    Note that for $i\in I_\circ$ we have
    \[
        \gr(K_i^{-1}E_iF_j)
        = \delta_{ij}(q_i-q_i^{-1})^{-1} + q_i^{\langle h_i,\alpha_j\rangle} \gr(F_j K_i^{-1}E_i),
    \]
    so that $P(S(\uq^+)\uq^-)\subset k[T]$.

    Define
    \[
        p' := \sum_{\alpha\in Q^+} P(b_{\alpha}^*b_{-\alpha})(0) e^{-\alpha}
    \]
    ($P(b_{-\alpha}^*b_{-\alpha})\in k[T]$, which we evaluate in 0).
    Then we have
    \begin{align*}
        \widehat{\Upsilon}(v_\lambda b^*\otimes bv_\lambda)
        &= \sum_{\alpha\in Q^+} \Upsilon(v_\lambda b_{\alpha}^*\otimes
        b_{-\alpha}v_\lambda)\\
        &= \sum_{\alpha\in Q^+} v_\lambda^*(b_\alpha^* b_{-\alpha}v_\lambda) e^{\lambda-\alpha}\\
        &= \sum_{\alpha\in Q^+} v_\lambda^*(P(b_\alpha^*b_\alpha)v_\lambda)e^{\lambda-\alpha}\\
        &= e^\lambda p',
    \end{align*}
    where $\gr(\widetilde{\uq}S(\uq^+)_+)v_\lambda=0$
    and $v_\lambda^*\gr(\uq^-_+\widetilde{\uq})=0$, which explains
    the occurrence of $P$, and where $fv_\lambda=f(0)v_\lambda$ for
    $f\in k[T]$.

    It now only remains to see that $p'$ is supported in $Q(\Sigma)^-$.
    For this we note that $bv_\lambda$ can only be invariant under
    $\gr(k[T])\subset\gr(\widetilde{\uqb})$ if it is supported in
    $-\Theta$-invariant weights.
    Consequently, $b_{-\alpha}v_\lambda\ne0$ (and hence without loss of generality $b_{-\alpha}\ne0$) only if
    $\lambda-\alpha$ is $-\Theta$-invariant.
    Since $\lambda$ has the same property, we obtain the same for
    $\alpha$, and hence $\alpha=\tilde{\alpha}\in Q(\Sigma)$.
\end{proof}

Having now constructed the restriction $p'$ of our universal
ZSF of the associated graded rings, our next goal is to use $p'$ to
determine the leading function of $\Rad'(x)$ for elements
$x\in Z(\widetilde{\uq})$.

\begin{proposition}\label{prop-gamma-res-gr}
    Let $x\in Z(\widetilde{\uq})$.
    If $\gamma(x) = K_{2\mu} + \lot$, then the following hold:
    \begin{enumerate}
        \item there is a rational function $a_\mu\in k(Q(\Sigma))$ such that
        $\Res'(x) = a_\mu T_{2\mu} +\lot$;
        \item $x\in\mathcal{F}(\widetilde{\uq})_{2\mu}$, so that
        we can write $\tilde{x} = x + \mathcal{F}_{<2\mu}(\widetilde{\uq})\in\gr(\widetilde{\uq})$;
        \item $\gr(\gamma)(\tilde{x}) = K_{2\mu}$;
        \item $\gr(\Res')(\tilde{x}) = a_\mu T_{2\mu}$.
    \end{enumerate}
\end{proposition}
\begin{proof}
    Claim~(i) follows from \cite[Lemma~5.8]{Le08}.
    We can write
    \[
        x = \sum_{\tilde{\lambda}\le\mu} a_\lambda c_\lambda
    \]
    where $c_\lambda$ is the distinguished element
    \[
        c_\lambda \in (K_{2\lambda} + \adr(\uq_+)(K_{2\lambda}))
        \cap Z(\widetilde{\uq}).
    \]
    Since the results are additive, we can henceforth assume
    without loss of generality
    that $x=a_\lambda c_\lambda$ with $\tilde{\lambda}\le\mu$.
    By Lemma~\ref{lem-adr-invariance-of-filtration}, we can conclude
    that $x \in \mathcal{F}_{2\mu}(\widetilde{\uq})$ and hence Claim~(ii).
    For the last two claims let $h\in A$ and expand
    \[
        K_h x = \sum_j b'_j(h) K_{h+\mu_j}b_j(h)
    \]
    for elements $b'_j(h)\in\widetilde{\uqb}'$, $b_j(h)\in\widetilde{\uqb}$, and weights $\mu_j\in A$.
    Since every generator of $\widetilde{\uqb}$ and
    $\widetilde{\uqb}'$ lies in $\widetilde{\uq}_0\oplus\widetilde{\uq}_{-\tilde{\alpha}_i}$ for some
    $i\in I_\circ$ with a nonzero first component, we have
    $\mu_j\le 2\mu$.
    We thus obtain that
    \[
        K_h \tilde{x} = \sum_{\substack{j\\\mu_j=2\mu}}
        \widetilde{b'_j(h)} K_{h+2\mu} \widetilde{b_j(h)},
    \]
    so that
    \[
        \gr(\Res')(\tilde{x}) = 
        \sum_{\substack{j\\\mu_j=2\mu}}
        \epsilon(\widetilde{b'_j(h)})
        \epsilon(\widetilde{b_j(h)}) T_{2\mu},
    \]
    where we interpret $\epsilon(\widetilde{b'_j(h)})$ to
    mean $\epsilon(b'_j(h))$ (by Lemma~\ref{lem-B-embedded}
    this is well-defined).
    This therefore equals
    \[
        \sum_{\substack{j\\\mu_j=2\mu}} \epsilon(b'_j(h))
        \epsilon(b_j(h)) T_{2\mu}
        = a_\mu T_{2\mu}.
    \]
    Moreover, by \cite[Lemma~5.7]{Le08}, the constant term of
    $a_\mu$ is 1.
    With an analogous proof we can show loc.cit. in the associated
    graded setting, showing that $\gr(\gamma)(\tilde{x})=K_{2\mu}$.
\end{proof}

\begin{lemma}\label{lem-centre-leading-terms}
    Let $x\in Z_{\widetilde{\uq}}(\widetilde{\uqb})\cap\mathcal{F}_{2\mu}(\widetilde{\uq})$.
    Write $\tilde{x}\in\gr(\widetilde{\uq})$ for the associated element.
    If $\gr(\gamma)(\tilde{x}) = K_{2\mu}$, then
    $\gr(\Rad')(\tilde{x}) = p' T_{2\mu} p^{\prime-1}$.
\end{lemma}
\begin{proof}
    On one hand we have
    \[
        \widehat{\Upsilon}(v_\lambda b^*\otimes \tilde{x}bv_\lambda)
        = \gr(\Rad')(\tilde{x}) \widehat{\Upsilon}(v_\lambda b^*\otimes bv_\lambda)
        = \gr(\Rad')(\tilde{x}) p' e^\lambda.
    \]
    On the other hand, we have
    \[
        \widehat{\Upsilon}(v_\lambda b^*\otimes \tilde{x}bv_\lambda)
        = q^{2\mu\cdot\lambda}
        \widehat{\Upsilon}(v_\lambda b^*\otimes bv_\lambda)
        = q^{2\mu\cdot\lambda} p' e^\lambda
        = p' T_{2\mu} p^{\prime-1} p' e^\lambda
    \]
    We thus conclude that $\gr(\Rad')(\tilde{x})$ and $p' T_{2\mu}p^{\prime-1}$ coincide on the span of $p' e^\lambda$ for all
    $\lambda\in (X\otimes_\Z\Q)^{-\Theta}$.
    Since $\gr(\Rad')(\tilde{x})$ can be written in terms of $K_h$ for $h\in A$,
    this is enough to show equality.
\end{proof}

\begin{theorem}\label{thm-initial-term-given-by-p}
    There exists a power series $p'=1+\lot \in k[[e^{-\tilde{\alpha}_i}\mid i\in I_\circ]]$ (see Lemma~\ref{lem-definition-p'}) satisfying
    the following property:
    For any $x\in Z_{\widetilde{\uq}}(\widetilde{\uqb})$ with
    $\gamma(x) = K_{\mu}+\lot$ we have
    $\Rad'(x) = p' T_{\mu} p^{\prime-1}+\lot$.

    Moreover, writing $p:= p'(\cdot-\rho)$, we get
    $\Rad(x) = p T_{\mu}p^{-1} + \lot$.
\end{theorem}
\begin{proof}
    Write
    \[
        \Rad'(x) = \sum_{\nu\le \mu} a_\nu T_\nu
    \]
    (possible by \cite[Lemma~5.8]{Le08}).
    For every $\lambda\in X^+$, an appropriately scaled version
    $ac_\lambda$ of the Casimir element from the proof of
    Proposition~\ref{prop-gamma-res-gr} satisfies the following:
    $\gamma(a c_\lambda) = K_{2\tilde{\lambda}} +\lot$ and
    $\Rad'(a c_\lambda) = p' T_{2\tilde{\lambda}}p^{\prime-1}+\lot$
    by Proposition~\ref{prop-gamma-res-gr} and
    Lemma~\ref{lem-centre-leading-terms}.

    We write $\Rad'(ac_\lambda)=\sum_{\nu\le 2\tilde{\lambda}} b_\nu T_\nu$
    for appropriate rational functions $b_\nu\in k(Q(\Sigma))$.
    Note that $b_{2\tilde{\lambda}} = \frac{p'}{p'(\cdot + 2\tilde{\lambda})}$
    Since $c_\lambda$ and $x$ commute, so do their radial parts,
    hence
    \[
        \sum_{\nu_1\le \mu}\sum_{\nu_2\le 2\lambda}
        a_{\nu_1} T_{\nu_1} b_{\nu_2} T_{\nu_2}
        = \sum_{\nu_1\le\mu} \sum_{\nu_2\le2\lambda} b_{\nu_2}T_{\nu_2} a_{\nu_1} T_{\nu_1}.
    \]
    By only considering the coefficients of $T_{\mu+2\tilde{\lambda}}$, we
    obtain the following equation
    \[
        \frac{a_\mu p'(\cdot+\mu)}{p'(\cdot +\mu+2\tilde{\lambda})}=a_\mu b_{2\tilde{\lambda}}(\cdot+\mu) =
        b_{2\tilde{\lambda}} a_\mu(\cdot + 2\tilde{\lambda})
        = \frac{p' a_\mu(\cdot+2\tilde{\lambda})}{p'(\cdot+2\tilde{\lambda})}.
    \]
    Reordering yields
    \[
        \frac{a_\mu p'(\cdot+\mu)}{p'} = T_{2\tilde{\lambda}} \frac{a_\mu p'(\cdot+\mu)}{p'}.
    \]
    This holds for all $\lambda\in X^+$, hence the Laurent series $\frac{a_\mu p'(\cdot+\mu)}{p'}$ is periodic, hence constant.
    Since $p'$ and any shift thereof has constant term 1, as does $a_\mu$,
    we conclude $a_\mu = \frac{p'}{p'(\cdot+\mu)}$, i.e.
    $a_\mu T_\mu = p' T_\mu p^{-1}$.
\end{proof}

Having thus established that there is a uniform manner in which the
leading term of $\Rad'(x)$ and $\Rad(x)$ of an element $x\in Z_{\widetilde{\uq}}(\widetilde{\uqb})$ depends on $\gamma(x)$,
we have progressed a lot.

We now take a closer look at the power series $p', p$ that everything
depends on.
In particular, we make sense of it using a rank-1 reduction and knowledge
about the rank 1 cases that can be gleaned from the literature.

\section{Rank-1 Reduction}\label{sec-rk1-reduction}
For the remainder of this section let $i\in I_\circ$ and
write $I[i]$ for the union of the connected components of $I_\bullet\cup\set{i,\tau(i)}$ containing $i$ and $\tau(i)$.
In other words: $I[i]$ is the subdiagram containing $i$, $\tau(i)$
and any black vertices connected to them.
$I[i]$ is $\tau$-invariant, and we can associate a quantum symmetric
pair of rank 1 to it: write $\uq[i]$ for the Levi subalgebra of $\uq$
associated to $I[i]$, i.e. the algebra generated by $\uq^0$ and
$E_j$, $F_j$ for $j\in I[i]$.
Write $\uqb[i]$ and $\uqb'[i]$ for the subalgebras of $\uqb,\uqb'$
generated by $k[Y^\Theta]$, $E_j, F_j$ for $I[i]\cap I_\bullet$, and
$B_i$ and $B_{\tau(i)}$.

\begin{definition}
    $(\uq[i],\uqb[i])$ and $(\uq[i],\uqb'[i])$ are the rank-1 QSPs
    associated with $I[i]$.
    We add the suffix $[i]$ to a lot of previous constructions to indicate
    that they were performed for these QSPs instead of $(\uq,\uqb)$ or
    $(\uq,\uqb')$.
\end{definition}

\begin{lemma}\label{lem-rank1-iso}
    Let $\lambda\in (X\otimes_\Z\Q)^{-\Theta}$ and let
    $0\ne w\in \overline{M}(\lambda)_\mu$ be a weight vector annihilated by
    $\uq_\bullet[i]_+$, $k[T]_+$,
    and $S(\uq[i]^+)_+$, then
    $\mu\in (X\otimes_\Z\Q)^{-\Theta}$ and there is a unique
    $\gr(\widetilde{\uq}[i])$-homomorphism
    $\overline{M}[i](\mu)\to\overline{M}(\lambda)$ mapping
    $v_\mu\mapsto w$.
    This homomorphism is injective.
\end{lemma}
\begin{proof}
    Let $\nu\in Y^\Theta$, then $0=(K_\nu-1)w=(q^{\langle \nu,\mu\rangle}-1)w$, consequently $\langle\nu,\mu\rangle=0$.
    This shows that $\nu\in (X\otimes_\Z\Q)^{-\Theta}$.

    We now consider $\overline{M}[i](\mu)$.
    It has $v_\mu$ as cyclic vector, consequently a $\uq[i]$-homomorphism
    to $\overline{M}(\lambda)$ is completely determined by what it
    maps $v_\lambda$ to.
    This shows uniqueness.

    For existence, we note that $v_\lambda$ is annihilated by
    $S(\uq[i]^+)_+$ and $\uq_\bullet[i]_+$ and has weight $\lambda$
    (and nothing further).
    The same is true for $w$, therefore there is such a morphism.

    For injectivity, we note that $\overline{M}[i](\lambda)$
    is free over $\mathcal{N}^-[i](\mu)$, as is $\overline{M}(\lambda)$.
\end{proof}

Next we show that $b$ (resp. $b^*$) factors through $b[i]$ (resp.
$b^*[i]$).
For that we first prove a technical lemma

\begin{lemma}\label{lem-recursive-step}
    Let $\lambda\in (X\otimes_\Z\Q)^{-\Theta}$, let
    $w\in\widehat{\mathcal{N}}^-$ be such that
    $wv_\lambda$ is $\uqb[i]$-cyclic.
    Let $\nu$ be a maximal weight such that $w_\nu\ne0$,
    then $b[i]w_\nu v_\lambda$ is also $\uqb[i]$-cyclic.
\end{lemma}
\begin{proof}
    We show that $w_\nu v_\lambda$ satisfies the conditions of
    Lemma~\ref{lem-rank1-iso}.
    We first note that $w_\nu v_\lambda$ has weight $\nu+\lambda$,
    which is $-\Theta$ invariant.
    Let $j\in I_\bullet[i]$, then $E_iK_i w_\nu v_\lambda$ is
    the $\nu+\lambda+\alpha_i$-weight component of $E_iK_i wv_\lambda$,
    which is 0.
    Let $j=i$ or $\tau(i)$, then $E_jK_j w_\nu v_\lambda$ is
    the $\nu+\lambda+\alpha_j$ component of
    \[
        \bm{c}_{\tau(j)}^{-1}T_{w_\bullet}^{-1}(B_{\tau(j)}K_{\tau(j)})
        K_j
        wv_\lambda
        = \frac{s_{\tau(j)}}{\bm{c}_{\tau(j)}} wv_\lambda
    \]
    which is 0.
    Consequently, $w_\nu v_\lambda$ is annihilated by $S(\uq[i]^+)_+$.

    Moreover, $w_\nu v_\lambda$ generates a finite-dimensional
    $\uq_\bullet[i]$-module, and is a highest-weight vector of weight
    $0$.
    Consequently, $w_\nu v_\lambda$ is annihilated by
    $\uq_\bullet[i]_+$.
    We can now apply Lemma~\ref{lem-rank1-iso} and find that
    $w_\nu v_\lambda$ is the highest-weight vector of a $\uq[i]$-submodule isomorphic
    to $\overline{M}[i](\nu+\lambda)$.
    As a result $b[i] w_\nu v_\lambda$ is $\uqb[i]$-cyclic.
\end{proof}

\begin{proposition}\label{prop-existence-w}
    There exists $w\in\widehat{\mathcal{N}}^-$ such that
    for every $\lambda\in (X\otimes_\Z\Q)^{-\Theta}$ every nonzero weight
    component of $wv_\lambda$ satisfies the conditions of
    Lemma~\ref{lem-rank1-iso} and
    $b[i]wv_\lambda = bv_\lambda$.
\end{proposition}
\begin{proof}
    Since every $\overline{M}(\lambda)$ is a free $\mathcal{N}^-$-module,
    every $\widehat{M}(\lambda)$ is a free $\widehat{\mathcal{N}}^-$-module,
    so the equation $b[i]wv_\lambda=bv_\lambda$ implies
    $b[i]w = b$.
    From the construction of $w$, it will also become clear that it
    suffices to prove this statement for one choice of $\lambda$
    and it will follow for all others.
    
    Let $\lambda\in (X\otimes_\Z\Q)^{-\Theta}$ be arbitrary.
    We will now recursively define sets $X_n$ ($n\in\N_0$)
    of tuples $(\mu, w)$ where $\mu\in Q^-$ of height $-n$, and
    $w\in\overline{M}(\lambda)_{\lambda+\mu}$ such that
    $w$ is the highest-weight vector of a $\gr(\widetilde{\uq}[i])$-module isomorphic to $\overline{M}[i](\lambda+\mu)$ and such that
    \[
        bv_\lambda - \sum_{m\le n}\sum_{(\mu,w)\in X_m} b[i]w
    \]
    is supported in weights of height $<-n$ and is $\uqb[i]$-cyclic.

    For $n=0$, define $X_0:=\set{(0, v_\lambda)}$, then the
    claim follows from the fact that $b$ and $b[i]$ both have leading
    term 1 and the fact that $v_\lambda$ is a highest-weight vector
    even for $\gr(\widetilde{\uq})$.

    For $n>0$, we first define
    \[
        x:= bv_\lambda - \sum_{m<n}\sum_{(\mu,w)\in X_m} b[i]w
    \]
    and then define $X_n$ to be the set of $(\mu,x_{\lambda+\mu})$ where
    $\mu$ ranges over the maximal weights of height $n$ satisfying $x_{\lambda+\mu}\ne0$.
    By Lemma~\ref{lem-recursive-step}, $b[i]w$ is cyclic for
    any $(\mu,w)\in X_n$.
    Furthermore, $b[i]w = w + \lot$, so that
    \[
        x-\sum_{(\mu,w)\in X_n} b[i]w = bv_\lambda -
        \sum_{m\le n}\sum_{(\mu,w)\in X_m} b[i]w
    \]
    is supported in weights of height $<-n$.

    For each $(\mu,w)\in X_n$ there exists $w_\lambda\in\mathcal{N}^-_\mu$
    such that $w=w_\mu v_\lambda$.
    Then $w = (w_\lambda)_{\lambda\in Q^-}$ satisfies the claim.
\end{proof}

Analogously, we also obtain $w^*=(w^*_\lambda)_{\lambda\in Q^+}
\in\widehat{\mathcal{N}}^+$ such that $b^* = w^*b^*[i]$.

\begin{lemma}\label{lem-existence-ki}
    Let $k'[i]:= \sum_{\mu\in Q^-} P(w_{-\mu}^*w_\mu)(0)e^\mu$,
    then $k'[i]=1+\lot$ and $p'=p'[i]k'[i]$.
    ($P$ as in Lemma~\ref{lem-definition-p'}.)
\end{lemma}
\begin{proof}
    Note that $\gr(\widetilde{\uq}[i]S(\uq[i]^+)_+)$ annihilates $w_\mu v_\lambda$ and that $\gr(\uq[i]^-_+\widetilde{\uq}[i])$ annihilates
    $v_\lambda w^*_{-\mu}$ for all $\mu\in Q^-$.
    The same is true for $\gr(k[T]_+)$, so we have
    \begin{align*}
        \Upsilon(v_\lambda b^*\otimes bv_\lambda)
        &= \Upsilon(v_\lambda w^* b^*[i]\otimes b[i]wv_\lambda)\\
        &= \sum_{\mu\in Q^-}\sum_{\nu\in Q^-[i]}
        v^*_\lambda (w^*_{-\mu} b^*[i]_{-\nu} b[i]_\nu w_\mu v_\lambda)e^{\lambda+\mu+\nu}\\
        &= \sum_{\mu\in Q^-}\sum_{\nu\in Q^-[i]}
        v^*_\lambda (w^*_{-\mu} P(b^*[i]_{-\nu}b[i]_\nu)w_\mu v_\lambda)e^{\lambda+\mu+\nu}\\
        &= p'[i] \sum_{\mu\in Q^-} v_\lambda^*(w^*_{-\mu} w_\mu v_\lambda)
        e^{\lambda+\mu}\\
        &= p'[i] k'[i] e^\lambda
    \end{align*}
    analogously to the proof of Lemma~\ref{lem-definition-p'}.
\end{proof}

\begin{definition}
    Let $\uq^-[i]^\perp$ be the subalgebra of $\uq^-$ generated
    by
    \[
        \adr(\uq[i])(F_j)\qquad (j\not\in I[i]).
    \]
    Then $\uq^-\cong \uq^-[i]\otimes\uq^-[i]^\perp$.
\end{definition}

\begin{lemma}\label{lem-uminperp-weights}
    Let $\mu$ be a weight of $\uq^-[i]^\perp$, then
    \[
        \tilde{\mu}\in-\sum_{\alpha\in\Sigma^+\setminus\Q\Sigma[i]}\N_0\alpha.
    \]
\end{lemma}
\begin{proof}
    Since weights are additive, it suffices to show the claim
    for multiplicative generators of $\uq^-[i]^\perp$.
    $\uq^-[i]^\perp$ is generated by $\ad_r(x)(F_j)$ of weight
    $\mu\in R$.

    First note that $\mu\in -\alpha_j+R[i]$, hence $\mu$ cannot be 
    an element of $R[i]$.
    Secondly, since $R[i]$ is $\tau$-invariant and contains
    entire connected components of black dots, $R[i]$ is $\Theta$-invariant,
    as is $R\setminus R[i]$, hence $\Theta(\mu)\not\in R[i]$.
    Moreover, since $\Theta(\alpha_j) \in -\alpha_{\tau(j)}+Q(R_\bullet)$,
    $\Theta(\mu)\in \alpha_{\tau(j)}+Q(R_\bullet[i])$.

    We conclude $2\tilde{\mu} = \mu-\Theta(\mu)\in -\alpha_j-\alpha_{\tau(j)} + Q(R[i])$, which has empty intersection with
    $\Q\Sigma[i]$.
\end{proof}

\begin{proposition}
    Let $w\in\mathcal{N}^-$ be an element of weight $\mu$ such that 
    $wv_\lambda$ satisfies the conditions of Lemma~\ref{lem-rank1-iso},
    then
    \[
        \mu\in -\sum_{\alpha\in\Sigma^+\setminus\Q\Sigma[i]} \N_0 \alpha.
    \]
    In particular, $k'[i]\in k[[e^\beta\mid \beta\in\Sigma^-\setminus\Q\Sigma[i]]]$.
\end{proposition}
\begin{proof}
    Write
    \[
        w = \sum_k a_k b_k
    \]
    where $a_k\in\uq^-[i]$ and $b_k\in\uq^-[i]^\perp$.
    Without loss of generality let the $a_k,b_k$ be weight vectors,
    of weight $\lambda_k,\mu_k$,
    and the $b_k$ be linearly independent.

    Pick $\mu'$ to be a minimal element of $\set{\mu_k\where k}$ (with
    respect to the dominance ordering), then $\widetilde{\mu'}\in -\sum_{\alpha\in\Sigma^+\setminus[i]}\N_0\alpha$ by Lemma~\ref{lem-uminperp-weights}.
    It suffices to show that $\mu=\mu'$.

    By assumption, $wv_\lambda$ is an $S(\uq^+[i])$-highest-weight vector, i.e. for $j\in I[i]$ we have
    \[
        K_j^{-1}E_j wv_\lambda = 0.
    \]
    Recall from Lemma~\ref{lem-isomorphic-suplus-modules} that
    \[
        0 = K_j^{-1}E_jwv_\lambda = (K_j^{-1}E_j\ast w)v_\lambda,
    \]
    hence by freeness that $K_j^{-1}E_j\ast w=0$ for all $j\in I[i]$.
    Note that
    \begin{align*}
        \adr(E_j)(w) &= \sum_k \adr(E_j)(a_kb_k)\\
        &= \sum_k \big(\ad_r(E_j)(a_k)b_k\\
        &\qquad + \ad_r(K_j)(a_k)\ad_r(E_j)(b_k)\big)
        \in \uq^-[i]\otimes \uq^-[i]^\perp.
    \end{align*}
    The terms that have weight $\mu'$ in the second tensor component
    are
    \[
        \sum_{\substack{k\\\mu_k=\mu'}} \ad_r(E_j)(a_k)b_k
    \]
    (since $\mu'$ is minimal, it cannot be written as
    $\mu_k+\alpha_j$ for any $k$).
    By linear independence, we have $\ad_r(E_j)(a_k)=0$ for all
    $k$ such that $\mu_k=\mu'$.
    Since $a_k\in\uq[i]^-$, we conclude that $a_k$ is a scalar, and
    hence that $\mu=\mu'$.

    For the last conclusion, we note that every weight vector of
    $w$ from Proposition~\ref{prop-existence-w} satisfies the
    conditions of Lemma~\ref{lem-rank1-iso}, so that the claim follows
    from the definition of $k'[i]$ in Lemma~\ref{lem-existence-ki}.
\end{proof}

\begin{lemma}
    Write $k[i]:= k'[i](\cdot-\rho)$, then
    $r_i k[i] = k[i]$.
\end{lemma}
\begin{proof}
    Pick $\nu\in X^+$ such that
    \[
        \tilde{\alpha}_j\cdot\nu \in\begin{cases}
            \set{0} & j\in I[i]\\
            \Q^\times & j\not\in I[i]
        \end{cases}
    \]
    for all $j\in I_\circ$, e.g.
    by taking an appropriate sum of fundamental weights.
    Then $T_{2\nu}$ commutes with $p[i]$, and the only elements of
    $k[[e^\beta\mid\beta\in\Sigma^-\setminus\Q\Sigma[i]]]$ it commutes
    with are constant functions.
    Let $x=a c_\nu$, then by Theorem~\ref{thm-initial-term-given-by-p},
    we have
    \[
        \Rad(x) = \sum_{w\in W_\Theta} w(p T_{2\tilde{\nu}}p^{-1})
        + \sum_{W_\Theta\tilde{\kappa}<\tilde{\nu}}
        k(Q)T_{2h_{\tilde{\kappa}}}.
    \]
    Note that $T_{2h_{\tilde{\nu}}}$ commutes with $p[i]$, so we have
    \[
        p T_{2h_{\tilde{\nu}}} p^{-1}
        = k[i] T_{2h_{\tilde{\nu}}} k[i]^{-1}.
    \]
    Note that $r_i\tilde{\nu}=\tilde{\nu}$, so we conclude
    $r_i (k[i]T_{2h_{\tilde{\nu}}} k[i]^{-1})
    = k[i]T_{2h_{\tilde{\nu}}} k[i]^{-1}$.
    This implies that $\frac{r_i k[i]}{k[i]}$ is a power series in
    $k[[e^\beta\mid \beta\in\Sigma^-\setminus\Q\Sigma[i]]]$ that
    commutes with $T_{2\tilde{\nu}}$, and is hence a constant.
    Since both $k[i]$ and $r_ik[i]$ have constant term 1, that constant is
    one.
\end{proof}

\begin{definition}
    Write $\Sigma^{\textrm{ind}}$ for the indivisible subsystem of
    $\Sigma$, i.e. the union of Weyl orbits of all $\tilde{\alpha}_i$
    for $i\in I_\circ$.
    
    For $\alpha\in\Sigma^{\mathrm{ind}}$, say $\alpha=w\tilde{\alpha}_i$
    (for $w\in W_\Theta$ and $i\in I_\circ$) write $p_{\alpha}:= wp[i]$.
\end{definition}

Note that $p_\alpha$ does not depend on the choice of $w$ as $p[i]\in k[[e^{-\tilde{\alpha}_i}]]$, and hence $p[i]$ is stabilised by the stabiliser of
$\tilde{\alpha}_i$ in $W_\Theta$.
It also does not depend on the choice of $i$ as $I[i]$ and $I[j]$ are
isomorphic whenever $\tilde{\alpha}_i\in W_\Theta\tilde{\alpha}_j$,
and since any possible parameter dependence is through parameters
$\bm{c}_i,\bm{c}_{\tau(i)}$ or $\bm{s}_i$ for $i\in I_\circ$ for which
$\tilde{\alpha}_j\in W_\Theta\tilde{\alpha}_i$ implies that $i=j$ or
$i=\tau(j)$.

\begin{proposition}\label{prop-p-function}
    $p$ as in Theorem~\ref{thm-initial-term-given-by-p} is given by
    \[
        p = \prod_{\alpha\in\Sigma^{\mathrm{ind},+}} p_\alpha.
    \]
\end{proposition}
\begin{proof}
    Write $pP = \prod_{\alpha\in\Sigma_1^+}p_\alpha$ for a power series
    $P\in k[[e^{-\tilde{\alpha}_i}\mid i\in I_\circ]]$, which then
    also satisfies $p'(0)=1$.

    Define
    \[
        K_i := \prod_{\alpha\in\Sigma^{\mathrm{ind},+}\setminus\set{\tilde{\alpha}_i}} p_\alpha,
    \]
    then $K_i p[i] = pP = k[i] p[i] P$, which shows that
    $K_i = k[i] P$.
    Since both $K_i$ and $k[i]$ are invariant under $r_i$, the same
    must be true for $P$.

    As a consequence, $P$ is invariant under $W^\Sigma$, which implies
    that it must be a constant, namely 1.
\end{proof}

We have thus successfully shown that $p$, and hence the leading
term of the radial part of any element of $Z_{\widetilde{\uq}}(\widetilde{\uqb})$, can be recovered easily from the
functions $p[i]$ for $I_\circ$, i.e. by knowing $p$ for 
all rank-1 subdiagrams of $I$.

\section{Computation in Rank 1}\label{sec-rk1-computation}
We now compute $p$ for the Satake diagrams of rank 1.
Note that $p\in k[[\Sigma^-]]$ is independent of the specific lattices $X,Y$ used to define
$\uq,\uqb,\uqb'$.
As such, any $p[i]$ from Section~\ref{sec-rk1-reduction} can
be computed by considering a QSP of type $I[i]$
in isolation, which we are now going to do.

Assume that
$(\uq,\uqb,\uqb')$ is of a rank-1 type, i.e. one
of
\begin{align*}
    \diag(\mathsf{A}_1), \quad \mathsf{AI}_1, \quad \mathsf{AII}_3,\quad \mathsf{AIV}_n\: (n\ge2),\\\mathsf{BII}_n\: (n\ge2),\quad \mathsf{CII}_{n,1}\:(n\ge3),\quad \mathsf{DII}_n\:(n\ge4),\quad \mathsf{FII}.
\end{align*}

\subsection{Letzter Cases}
We consider the case where $(\uq,\uqb,\uqb')$ is of type
$\diag(\mathsf{A}_1)$, $\mathsf{AII}_3$, $\mathsf{BII}$, or
$\mathsf{DII}$.
In these cases there exist a minuscule weight $\mu\in P$ such that
$2\widetilde{\mu} = \widetilde{\alpha}_i$ and
$2\widetilde{w_0\mu} = -\widetilde{\alpha}_i$, where $i\in I_\circ$ is the
unique white node: for $\diag(\mathsf{A}_1)$ choose
$\omega_1$, for $\mathsf{AII}_3$ choose $\omega_1$, for
$\mathsf{BII}_n$ choose $\omega_n$, for $\mathsf{DII}_n$ choose
$\omega_{n-1}$ or $\omega_n$.

\begin{lemma}[{\cite[Theorem~4.7]{Le04}}]\label{lem-letzter-radial-part}
    Let $\mu\in X^+$ correspond to one of the weights from above
    (if it exists), then for an appropriate constant $a\in k^\times$ we
    have
    \[
        \Rad(ac_\mu) =
        \frac{q^{\rho\cdot \tilde{\alpha}_i}e^{\tilde{\alpha}_i}
        - q^{-\rho\cdot\tilde{\alpha}_i}e^{-\tilde{\alpha}_i}}{e^{\tilde{\alpha}_i}-e^{-\tilde{\alpha}_i}}T_{2\tilde{\mu}}
        + \frac{q^{-\rho\cdot\tilde{\alpha}_i}e^{\tilde{\alpha}_i}
        - q^{\rho\cdot\tilde{\alpha}_i}e^{-\tilde{\alpha}_i}}{e^{\tilde{\alpha}_i}-e^{-\tilde{\alpha}_i}}T_{2\widetilde{w_0\mu}}
    \]
    provided that
    \[
        \gamma(ac_\mu)
        = q^{\rho\cdot\tilde{\alpha}_i} K_{2\tilde{\mu}}
        + q^{-\rho\cdot\tilde{\alpha}_i}K_{2\widetilde{w_0\mu}}.
    \]
\end{lemma}

\begin{lemma}\label{lem-letzter-p-function}
    We have
    \[
        p = \frac{\qty(q^{-2\rho\cdot\tilde{\alpha}}e^{-2\tilde{\alpha}_i};q_{2\tilde{\alpha}_i}^{-1})_\infty}{\qty(e^{-2\tilde{\alpha}_i};q_{2\tilde{\alpha}_i}^{-1})_\infty}
    \]
    where $q_\alpha:= q^{\abs{\alpha}^2/2}$ for $\alpha\in X\otimes\Q$.
\end{lemma}
\begin{proof}
    By Theorem~\ref{thm-initial-term-given-by-p} and Lemma~\ref{lem-letzter-radial-part}
    we have
    \[
        \frac{p}{p(\cdot+2\tilde{\mu})}
        = \frac{1-q^{-2\rho\cdot\tilde{\alpha}_i}e^{-2\tilde{\alpha}_i}}{1-e^{-2\tilde{\alpha}_i}}.
    \]
    Note that $2\tilde{\mu}\cdot 2\tilde{\alpha}_i = 2\abs{\tilde{\alpha}}^2$,
    so that
    \[
        p = \frac{\qty(q^{-2\rho\cdot\tilde{\alpha}}e^{-2\tilde{\alpha}_i};q_{2\tilde{\alpha}_i}^{-1})_\infty}{\qty(e^{-2\tilde{\alpha}_i};q_{2\tilde{\alpha}_i}^{-1})_\infty}.\qedhere
    \]
\end{proof}

\subsection{Non-Letzter Cases}
There remain four cases to treat:
$\mathsf{AI}_1$, which accounts for all non-standard
parameters, $\mathsf{AIV}_n$, which accounts for all
standard parameters, $\mathsf{CII}_{n,2}$, and
$\mathsf{FII}$.
In all cases except $\mathsf{AI}_1$ we have
$\Sigma = \set{\pm\tilde{\alpha}_i,\pm2\tilde{\alpha}_i}$.
We also define $\Sigma^*:=\Sigma$ then.
For $\mathsf{AI}_1$ we have $\Sigma=\set{\pm\tilde{\alpha}_i}$, and we define
$\Sigma^* = \frac{1}{2}\Sigma\cup\Sigma$.

\begin{lemma}\label{lem-rk1-nonletzter-macdonald}
    Assume that $(\uq,\uqb,\uqb')$ is not of type $\mathsf{FII}$.
    Write $\beta\in\Sigma^*$ for the long root
    of $\Sigma^*$ (i.e. $2\tilde{\alpha}_i$ or $\tilde{\alpha}_i$),
    which is also the fundamental weight of $\Sigma$ and of $\Sigma^*$.

    Let $\mu\in X^+(0)$, say
    \[
        \langle (2\beta)^\vee, \mu\rangle = \frac{\beta\cdot\mu}{\abs{\beta}^2} = m\in\N_0.
    \]    
    Let $\phi\in E^\gamma_{\uqb',\uqb}(\mu)$,
    then there is $C\in k$ such that $\Res(\phi) = Ce^{\mu-m\beta}P_m$,
    where $P_m$ is the $m$-th symmetric
    Askey--Wilson polynomial in $\frac{e^{\beta} + e^{-\beta}}{2}$ with base
    $q_\beta^2 := q^{\abs{\beta}^2}$ and parameters
    \[
        (a,b,c,d) = \begin{cases}
            (u_1u_2^{-1} q_\beta,\:-u_1u_2q_\beta,\:u_1^{-1}u_2 q_\beta,\:
            -u_1^{-1}u_2^{-1}q_\beta) & \mathsf{AI}_1\\
            \qty(\frac{\bm{c}_i}{\bm{d}_i}q_\beta^{n+1},
            (-1)^{n-1}\frac{\bm{c}_i}{\bm{d}_{\tau(i)}}q_\beta, \frac{\bm{c}_{\tau(i)}}{\bm{d}_{\tau(i)}}q_\beta^{n+1}, (-1)^{n-1}\frac{\bm{c}_{\tau(i)}}{\bm{d}_i}q_\beta)
            &
            \mathsf{AIV}_n\\
            (q_\beta^{2n-1},-q_\beta^3,q_\beta^{2n-3},-q_\beta) & \mathsf{CII}_{n,2}.
        \end{cases}
    \]
    where
    \[
        t_i=[u_1]_{q_\beta},\: s_i = [u_2]_{q_\beta}
    \]
    in the first case.
\end{lemma}
\begin{proof}
    This follows from \cite[Proposition~5.4]{qmsf}, which reformulates (for us relevant) results by
    Tom Koornwinder, Masatoshi Noumi, Mathijs Dijkhuizen, and Tetsuya Sugitani.
\end{proof}

\begin{conjecture}\label{conj-FII}
    Lemma~\ref{lem-rk1-nonletzter-macdonald} also holds for $\mathsf{FII}$, for parameters
    \[
        (a,b,c,d) = (q_\beta^{11},-q_\beta^7,q_\beta^5,-q_\beta).
    \]
\end{conjecture}

\begin{remark}
    We shall now explain the choice of parameters in Conjecture~\ref{conj-FII}.

    First, we note that $\mathsf{CII}_{n,2}$ and $\mathsf{FII}$ are odd cases:
    for all other rank-1 diagrams we only have as
    many parameters as degrees of freedom, subject to multiplicity constraints.
    In other words, the $q\to1$ limit imposes constraints on the parameters.
    For the four Letzter cases, this fixes the single parameter.
    For $\mathsf{AI}_1$ and $\mathsf{AIV}_n$, this leaves two degrees of freedom,
    which are determined by the parameters of the coideal subalgebras.

    For $\mathsf{CII}_{n,2}$ and $\mathsf{FII}$, however, we do not have any such
    extra parameters, as neither case admits non-standard parameters, and all
    sets of standard parameters are equivalent.
    We therefore need to look for two other natural constraints to put on the
    parameters of $\mathsf{FII}$.

    In the parlance of \cite{Sto}, we write
    \[
        (a,b,c,d) = (\kappa_{a_1}\kappa_{2a_1},
        -\kappa_{a_1}\kappa_{2a_1}^{-1},
        \kappa_{a_0}\kappa_{2a_0}q_\beta,-\kappa_{a_0}\kappa_{2a_0}^{-1}q_\beta).
    \]
    For $\mathsf{CII}_{n,1}$, we obtain the parameters
    \[
        \kappa=
        (\kappa_{a_1}, \kappa_{2a_1},\kappa_{a_0},\kappa_{2a_0})=(q_\beta^{n+1},q_\beta^{n-2},q_\beta^{n-2},q_\beta^{n-2}).
    \]
    Note that $\kappa_{2a_0}=\kappa_{a_0}$, suggesting that these parameters really belong to an affine root
    system smaller than
    $(C_1^\vee,C_1)$, namely not containing the orbit of
    $2a_0$.
    Furthermore, $\kappa_{2a_1}=\kappa_{a_0}$ implies
    that $\kappa^{\mathrm{d}}=\kappa$, i.e. the Macdonald data is self-dual.

    If we impose the same constraints on $\mathsf{FII}$, we obtain the four parameters
    \[
        \kappa = (q_\beta^9, q_\beta^2, q_\beta^2, q_\beta^2).
    \]
\end{remark}

\begin{lemma}\label{lem-nonletzter-shape-res}
    There exists $x\in Z_{\widetilde{\uq}}(\widetilde{\uqb})$ such that
    \[
        \gamma(x) = K_{\beta} + \lot
    \]
    and
    \[
        \Res(x) = \frac{(1-a^{-1}e^{-\beta})(1-b^{-1}e^{-\beta})(1-c^{-1}e^{-\beta})(1-d^{-1}e^{-\beta})}{(1-e^{-2\beta})(1-q_\beta^{-2}e^{-2\beta})}
        T_{\beta} + \lot
    \]
    with $a,b,c,d,q_\beta$ as in Lemma~\ref{lem-rk1-nonletzter-macdonald} and Conjecture~\ref{conj-FII}.
\end{lemma}
\begin{proof}
    Note that $\beta\in 2L$ is a pseudominuscule
    weight of $2\Sigma^*$.
    The existence of $x$ follows by Lemma~\ref{lem-gamma-surjective-morphism}.

    We have $\Rad(x)\Res(\phi)=\gamma(x)(\mu)\Res(\phi)$ for all
    $\phi\in E^{\epsilon}(\mu)$ and $\mu\in 2L$.
    Consequently, by Lemma~\ref{lem-rk1-nonletzter-macdonald}, $\Res(x)$ is a symmetric difference
    operator that has the functions
    \[
        e^{\mu-((2\beta)^\vee\cdot\mu)\beta}
        P_{(2\beta)^\vee\cdot\mu}((e^\beta+e^{-\beta})/2;a,b,c,d;q_\beta^2)
    \]
    for all $\mu\in 2L$ as eigenfunctions (note that $(2\beta)^\vee\cdot\mu=\frac{\beta\cdot\mu}{\abs{\beta}^2}$).
    By \cite[Lemma~5.8]{Le08}, $\Res(x)$ is supported
    in the root lattice of $\Sigma^*$.
    Consequently, it commutes with $e^{\mu-((2\beta)^\vee\cdot\mu)\beta}$ for
    $\mu\in 2L$, so that $\Res(x)$ has the Macdonald
    polynomials as eigenfunctions.

    Moreover, $\Res(x)$ is a symmetric difference operator,
    so it can be expressed in terms of
    Macdonald $Y$-operators.
    Due to the shape of the eigenvalues, we conclude
    \[
        \Res(x) = (q^{-\rho\cdot\beta} (Y+Y^{-1}) + \lot)_0,
    \]
    which amounts to the claimed operator
    by \cite[\S6.5.3]{Mac03}.
\end{proof}

\begin{lemma}\label{lem-nonletzter-p-function}
    We have
    \[
        p = \frac{\qty(a^{-1}e^{-\beta}, b^{-1}e^{-\beta},
        c^{-1}e^{-\beta},d^{-1}e^{-\beta};q_\beta^{-2})_\infty}{\qty(e^{-2\beta};q_\beta^{-2})_\infty}.
    \]
\end{lemma}
\begin{proof}
    This follows by Theorem~\ref{thm-initial-term-given-by-p} and
    Lemma~\ref{lem-nonletzter-shape-res}.
\end{proof}

\begin{remark}
    Note that if $a=-b=q^{\rho\cdot\beta}$ and $c=-d=q_\beta$, we obtain
    \[
        p = \frac{\qty(q^{-2\rho\cdot\beta}e^{-2\beta};q_{2\beta}^{-1})_\infty}{\qty(e^{-2\beta};q_{2\beta}^{-1})_\infty},
    \]
    which has the same shape as in Lemma~\ref{lem-letzter-p-function}.
    This is the case for standard cases (i.e. $\bm{s}=\bm{t}=0$) of type $\mathsf{AI}_1$, which agrees
    with \cite[Lemma~6.6]{Le03}.
\end{remark}

\section{Link with Macdonald}\label{sec-link-macdonald}
We now have enough information to compute $p$ for all cases except
$\mathsf{FII}$.
In this section we use this information to identify the zonal
spherical functions.

In arbitrary rank, define
\[
    \Sigma^* = \begin{cases}
        \Sigma & \mathcal{S}=\emptyset\\
        \Sigma\cup\frac{1}{2}W_\Theta\tilde{\alpha}_i & \mathcal{S}=\set{i}.
    \end{cases}
\]
This is again a root system.

From Lemmas~\ref{lem-letzter-p-function} and
\ref{lem-nonletzter-p-function}, we can conclude that
in the rank 1 case we have
\[
    p =  \frac{\qty(a^{-1}e^{-\beta}, b^{-1}e^{-\beta},
        c^{-1}e^{-\beta},d^{-1}e^{-\beta};q_\beta^{-2})_\infty}{\qty(e^{-2\beta};q_\beta^{-2})_\infty}
\]
where $\beta$ is the highest root of $\Sigma^*$
and where
\begin{equation}\label{eq-abcd-list}
    (a,b,c,d) = \begin{cases}
        \qty(q^{\rho\cdot\beta}, -q^{\rho\cdot\beta},q_\beta,-q_\beta) &
        \Sigma^*\text{ reduced}\\
        \qty(u_1u_2^{-1} q_\beta,\:-u_1u_2q_\beta,\:u_1^{-1}u_2 q_\beta,\:
        -u_1^{-1}u_2^{-1}q_\beta) & \mathsf{AI}_1\\
        \qty(\frac{\bm{c}_i}{\bm{d}_i}q_\beta^{n+1},
            (-1)^{n-1}\frac{\bm{c}_i}{\bm{d}_{\tau(i)}}q_\beta, \frac{\bm{c}_{\tau(i)}}{\bm{d}_{\tau(i)}}q_\beta^{n+1}, (-1)^{n-1}\frac{\bm{c}_{\tau(i)}}{\bm{d}_i}q_\beta)
        &
        \mathsf{AIV}_n\\
        \qty(q_\beta^{2n-1},-q_\beta^3,q_\beta^{2n-3},-q_\beta) & \mathsf{CII}_{n,1}\\
        \qty(q_\beta^{11},-q_\beta^7,q_\beta^5,-q_\beta) & \mathsf{FII}
    \end{cases}
\end{equation}
where $\bm{t}_i=[u_1]_{q_i}$ and $\bm{s}_i = [u_2]_{q_i}$
(in particular, $\mathsf{AI}_1$ has the same shape as the reduced case
for $\bm{t}_i=\bm{s}_i=0$).
By Proposition~\ref{prop-p-function} we can therefore conclude that
\[
    p = \prod_{\beta\in\Sigma}
    \frac{\qty(k_\beta e^{-2\beta};q_{2\beta}^{-1})_\infty}{\qty(e^{-2\beta};q_{2\beta}^{-1})_\infty}
\]
(where $k_{w\tilde{\alpha}_i} = q^{2\rho\cdot\tilde{\alpha}_i}$)
if $\mathcal{S}=\emptyset$ and $I$ contains no rank-1 
subdiagrams of type $\mathsf{AIV}$, $\mathsf{CII}$, or 
$\mathsf{FII}$.

For all other cases, $\Sigma^*$ is of type
$\mathsf{BC}_r$.
We adopt the usual convention of calling roots
long, medium, and short (such that in $\mathsf{BC}_1$ there
are no medium roots).
We then have
\begin{align*}
    p = &\prod_{\beta\in \Sigma^{*,+,\mathrm{l}}}
    \frac{\qty(a^{-1}e^{-\beta}, b^{-1}e^{-\beta},
        c^{-1}e^{-\beta},d^{-1}e^{-\beta};q_\beta^{-2})_\infty}{\qty(e^{-2\beta};q_\beta^{-2})_\infty}\cdot\\
        &\prod_{\beta\in\Sigma^{*,+,\mathrm{m}}}
        \frac{\qty(k_\beta^{-1}e^{-2\beta};q_{2\beta}^{-1})_\infty}{\qty(e^{-2\beta};q_{2\beta}^{-1})_\infty}.
\end{align*}
Note that $\rho\cdot\tilde{\alpha}_i=\rho[i]\cdot\tilde{\alpha}_i$ and that
if $\alpha\in\Sigma^{*,\mathrm{m}}$ and
$\beta\in\Sigma^{*,\mathrm{l}}$ we have
$q_\beta^2 = q_{2\alpha}$, so that all $q$-Pochhammer symbols have the same base.

The parameters $a,b,c,d$ are listed in \eqref{eq-abcd-list} and
depend on the parameters
$\bm{c},\bm{s},\bm{d},\bm{t}$ as well as the properties
of the special rank-1 subdiagram (i.e. for $i\in\mathcal{S}$ or for which $\Sigma[i]$ is of type
$\mathsf{BC}_1$).

We now show how $p$ determines the entire difference
operator in many cases.
Recall the notations $\Res$ from Definition~\ref{def-res}, $\Rad$ from Definition~\ref{def-radial-part},
and $\gamma$ from Definition~\ref{def-hc}.

\begin{lemma}\label{lem-p-determines-DO}
    Let $\omega\in 2L$ correspond to a (pseudo)minuscule
    weight of $2\Sigma^*$ and let
    $x\in Z_{\widetilde{\uq}}(\widetilde{\uqb})$ such that
    $\gamma(x) = K_{\omega}+\lot$.
    Then
    \[
        \Rad(x) =
        \sum_{w\in W_\Theta/W_{\Theta,\omega}}
        w (p T_{\omega}p^{-1}(1-T_{\xi-\omega}))
        + \gamma(x)(0) T_\xi
    \]
    where $\xi\in\omega-Q(\Sigma^*)^+$ is the unique element
    satisfying $\beta\cdot\xi=0$ for $\beta\in\Sigma^*$.
\end{lemma}
\begin{proof}
    Write $\Res(x)=\sum_{\mu\in\omega-Q(\Sigma^*)^-}
    a_\mu T_\mu$.
    By Proposition~\ref{prop-p-function}, we have
    \[
        a_\omega = \frac{p}{p(\cdot+\omega)},
    \]
    so that $a_\omega T_\omega = p T_\omega p^{-1}$.
    Moreover, $\Res(x)$ is $W_\Theta$-invariant, so
    $a_{w\omega}$ is also already determined by this.

    Since $\omega$ is (pseudo)minuscule, this leaves
    at most one coefficient undetermined:
    the one of $\xi$.
    Note, however, that the counit $\epsilon$ is the
    elementary ZSF of the trivial representation,
    and its restriction is 1,
    so we have
    \[
        \Res(x)1 = \gamma(x)(0) = \sum_{\mu\in\omega-Q(\Sigma^*)^-}
    a_\mu
    = a_{\xi} + \sum_{w\in W_\Theta/W_{\Theta,\omega}}
    wa_\omega,
    \]
    which we can solve for $a_\xi$, so that
    \begin{align*}
        \Res(x) &= \sum_{w\in W_{\Theta}/W_{\Theta,\omega}}
        w (pT_\omega p^{-1})
        + \gamma(x)(0)T_\xi - \sum_{w\in W_\Theta/W_{\Theta,\omega}} (wa_\omega)T_\xi\\
        &= \sum_{w\in W_{\Theta}/W_{\Theta,\omega}}
        w (pT_\omega p^{-1}(1-T_{\xi-\omega}))
        + \gamma(x)(0)T_\xi.\qedhere
    \end{align*}
\end{proof}

\begin{theorem}\label{thm-zsf}
    Let $R_0=\Sigma^{*,\mathrm{ind}}$, let
    $\Delta_0$ be a set of simple roots of $R_0$ made up
    of multiples of $\tilde{\alpha}_i$ for $i\in I_\circ$
    and consider the initial data
    $(2R_0,2\Delta_0,\mathrm{t},2L, 2L)$ (cf.
    \cite[Def.~9.2.5]{Sto}).
    Let $\theta\in \Sigma^*$ be the highest short root and let
    $i\in I_\circ$ be such that $\theta\in W_\Theta\tilde{\alpha}_i$
    or $2\theta\in W_{\Theta}\tilde{\alpha}_i$.
    
    Let $\kappa$ be a multiplicity function 
    satisfying
    $\kappa_{2\tilde{\alpha}_j+mc} = \kappa_{2\tilde{\alpha}_i}=q^{2\rho\cdot\tilde{\alpha}_i}$ for $\Sigma^*[j]$ reduced
    or for $\Sigma[j]$ of type $\mathsf{AI}_1$ with $j\not\in\mathcal{S}$,
    and further (i.e. if $\Sigma^*$ is not reduced)
    \[
        (a,b,c,d)=(\kappa_{\theta}\kappa_{2\theta},
        -\kappa_{\theta}\kappa_{2\theta}^{-1},
        q_\beta \kappa_{a_0}\kappa_{2a_0},
        -q_\beta\kappa_{a_0}\kappa_{2a_0}^{-1})
    \]
    for $a_0 = \frac{\abs{\theta}^2}{2}c-\theta$ the non-linear simple
    root of the affine root system $R$ (as in \cite[(9.2.6)]{Sto}),
    $\beta\in\Sigma^*[i]$ the highest root (and the highest short root
    of $2\Sigma^*[i]$),
    and $(a,b,c,d)$ the parameters from \eqref{eq-abcd-list} for
    $(\uq[i],\uqb'[i],\uqb[i])$.

    Let $\phi\in E^{\epsilon,\epsilon}_{\uqb',\uqb}(\mu)$ for $\mu\in 2L$, then
    there exists a constant $C\in k$ such that
    \[
        \Res(\phi) = C P_\mu,
    \]
    the $\mu$-th symmetric Koornwinder--Macdonald polynomial
    for initial data $(2R_0,2\Delta_0,\mathrm{t},2L,2L)$ and
    multiplicity function $\kappa$.
\end{theorem}
\begin{proof}
    By \cite[Corollary~9.3.18]{Sto}, applied to $-\lambda$ for any
    minuscule or pseudominuscule weight $\lambda\in 2L$ we obtain
    \[
        D_{-\lambda} = \kappa^{-1}_{\tau(\lambda)}
        \sum_{w\in W_\Theta/W_{\Theta, \lambda}}
        w (c_{\tau(\lambda)}) (\tau(-w\lambda)_q -1) + C
    \]
    (with notation as in \cite{Sto}).
    Note that $\tau(-w\lambda)_q=T_{w\lambda}$ in our notation.
    If we can show that $c_{\tau(\lambda)}$ is proportional
    to $\frac{p}{p(\cdot+\lambda)}$, there is
    $x\in Z_{\widetilde{\uq}}(\widetilde{\uqb})$ such that
    $\Rad(x) = D_{-\lambda}$.
    Since $\Res(\phi)$ is diagonalised by all such $D_{-\lambda}$,
    it must be a symmetric Koornwinder--Macdonald polynomial
    by \cite[Remark~9.3.23]{Sto}.

    To show the claim, note that by definition
    \[
        c_{\tau(\lambda)} = \prod_{a\in R^{\mathrm{t},+}\cap \tau(-\lambda)(R^{\mathrm{t},-})} c_a
        = \prod_{a\in R^{\mathrm{t},+}\setminus \tau(-\lambda)(R^{\mathrm{t},+})} c_a.
    \]
    We also have
    \[
        R^{\mathrm{t},+} = \set{\alpha + \mu_\alpha^{\mathrm{t}}mc\where
        \alpha\in R_0, m\in\N}
        \cup R_0^+
    \]
    and
    \[
        \tau(-\lambda)(\alpha+\mu_\alpha^{\mathrm{t}} mc)
        = \alpha + \qty(\mu_\alpha^{\mathrm{t}}m + \lambda\cdot\alpha)c.
    \]
    We obtain
    \[
        R^{\mathrm{t},+}\setminus
        \tau(-\lambda)(R^{\mathrm{t},+})
        = \set{\alpha + \mu_\alpha^{\mathrm{t}}mc\where
        \alpha\in R_0^+, 0\le m<\lambda\cdot\alpha^\vee}
    \]
    Moreover, for $a=\alpha+\mu_\alpha^{\mathrm{t}}mc$ we have
    \[
        c_a =
        \kappa_a^2
        \frac{\qty(1-\kappa_a^{-1}\kappa_{2a}^{-1}e^{-\alpha}q_\alpha^{-m})\qty(1+\kappa_a^{-1}\kappa_{2a}e^{-\alpha}q_\alpha^{-m})}{1-e^{-2\alpha}q^{-2m}}.
    \]
    If $R$ is reduced, we obtain $\kappa_a=\kappa_{2a}$ for all $a\in R$
    and $\kappa_{\alpha+\mu_\alpha^{\mathrm{t}}mc}=\kappa_{\alpha}$,
    so that
    \begin{align*}
        c_{\tau(\lambda)} &= \kappa_{\tau(\lambda)}^2
        \prod_{\alpha\in R_0^+} \prod_{m=0}^{\lambda\cdot\alpha^\vee-1}
        \frac{1-\kappa_\alpha^{-2}e^{-\alpha}q_\alpha^{-m}}{1-e^{-\alpha}q_\alpha^{-m}}\\
        &= \kappa_{\tau(\lambda)}^2
        \prod_{\alpha\in R_0^+} \frac{\qty(\kappa_\alpha^{-2}e^{-\alpha};q_\alpha^{-1})_{\lambda\cdot\alpha^\vee}}{\qty(e^{-\alpha};q_\alpha^{-1})_{\lambda\cdot\alpha^\vee}}\\
        &= \kappa_{\tau(\lambda)}^2 \frac{p}{p(\cdot+\lambda)}.
    \end{align*}
    Meanwhile, for $R$ non-reduced, we have
    $c_{\tau(\lambda)}=\kappa_{\tau(\lambda)}^2\prod_{\alpha\in R_0^+} \tilde{p}_\alpha$
    for
    \[
        \tilde{p}_\alpha = 
        \prod_{m=0}^{\lambda\cdot\alpha^\vee-1}
        \frac{\qty(1-\kappa_{\alpha+\mu_\alpha^{\mathrm{t}}mc}^{-1}
        \kappa_{2\alpha+2\mu_\alpha^{\mathrm{t}}mc}^{-1}e^{-\alpha}q_\alpha^{-m})
        \qty(1+\kappa_{\alpha+\mu_\alpha^{\mathrm{t}}mc}^{-1}\kappa_{2\alpha+2\mu_{\alpha}^{\mathrm{t}}mc}e^{-\alpha}q_\alpha^{-m})}{1-e^{-2\alpha}q_\alpha^{-2m}}.
    \]
    If $\alpha\in R_0^+$ is long (so that $\alpha\in R$ is medium),
    $\tilde{p}_\alpha$ reduces to
    \[
        \tilde{p}_\alpha = 
        \frac{\qty(\kappa_{\alpha}^{-2}e^{-\alpha};q_\alpha)_{\lambda\cdot\alpha^\vee}}{\qty(e^{-\alpha};q_\alpha)_{\lambda\cdot\alpha^\vee}}.
    \]
    Otherwise, if $\alpha\in R_0^+$ is short,
    we have that $\mu\cdot\alpha^\vee$ is even, and we have
    \[
        \kappa_{\alpha+\mu_\alpha^{\mathrm{t}}mc} = \begin{cases}
            \kappa_\theta & m\in 2\Z\\
            \kappa_{a_0} & m\in 2\Z+1
        \end{cases}
    \]
    and similarly for $\kappa_{2\alpha+2\mu_\alpha^{\mathrm{t}}mc}$.
    Thus, we obtain
    \begin{align*}
        \tilde{p}_\alpha &= \prod_{m=0}^{\lambda\cdot (2\alpha)^\vee-1}
        \frac{\qty(1-\kappa_\theta^{-1}\kappa_{2\theta}^{-1}e^{-\alpha}q_\alpha^{-2m})\qty(1+\kappa_\theta^{-1}\kappa_{2\theta}e^{-\alpha}q_\alpha^{-2m})}{\qty(1-e^{-2\alpha}q_\alpha^{-4m})}\\
        &\qquad\qquad  \cdot \frac{\qty(1-\kappa_{a_0}^{-1}\kappa_{2a_0}^{-1}e^{-\alpha}q^{-2m-1})\qty(1+\kappa_{a_0}^{-1}\kappa_{2a_0}e^{-\alpha}q_\alpha^{-2m-1})}{\qty(1-e^{-2\alpha}q_\alpha^{-4m-2})}\\
        &= \frac{\qty(a^{-1}e^{-\alpha},b^{-1}e^{-\alpha},c^{-1}e^{-\alpha}, d^{-1}e^{-\alpha}; q_\alpha^{-2})_{\lambda\cdot (2\alpha)^\vee}}{\qty(e^{-2\alpha};q_\alpha^{-2})_{\lambda\cdot\alpha^\vee}}.
    \end{align*}
    Consequently, we obtain
    \begin{align*}
        c_{\tau(\lambda)} &= \kappa_{\tau(\lambda)}^2
        \prod_{\alpha\in R_0^{+,\mathrm{s}}}
        \frac{\qty(a^{-1}e^{-\alpha},b^{-1}e^{-\alpha},c^{-1}e^{-\alpha}, d^{-1}e^{-\alpha}; q_\alpha^{-2})_{\frac{1}{2}\lambda\cdot\alpha^\vee}}{\qty(e^{-2\alpha};q_\alpha^{-2})_{\lambda\cdot\alpha^\vee}}\\
        &\qquad \cdot
        \prod_{\alpha\in R^{+,\mathrm{l}}}
        \frac{\qty(\kappa_{\alpha}^{-2}e^{-\alpha};q_\alpha)_{\lambda\cdot\alpha^\vee}}{\qty(e^{-\alpha};q_\alpha)_{\lambda\cdot\alpha^\vee}}\\
        &= \kappa_{\tau(\lambda)}^2 
        \frac{p}{p(\cdot+\lambda)}.
    \end{align*}
    We have thus shown that $c_{\tau(\lambda)}$ is proportional to
    $\frac{p}{p(\cdot+\lambda)}$ in every case.
    And hence by the argument above, the claim follows.
\end{proof}

\begin{remark}
    For the cases listed in \cite[Remark~5.5]{qmsf}, that were not
    covered in previous literature, we obtain Koornwinder polynomials
    of rank $r$ and base $q_i^2$ (where $\mathcal{S}=\set{i}$) with the following parameters:
    \begin{align*}
        \mathsf{BI}_{n,2}: r=2, &\qty(q_i^{\sigma-\tau+1}, -q_i^{\sigma+\tau+1},
        q_i^{-\sigma+\tau+1}, -q_i^{-\sigma-\tau+1}, q_i^{2n-3})\\
        \mathsf{CI}_n: r=n, &\qty(q_i^{\sigma-\tau+1},-q_i^{\sigma+\tau+1}, q_i^{-\sigma+\tau+1},
        -q^{-\sigma-\tau+1}, q_i)\\
        \mathsf{DI}_{n,2}: r=2, &\qty(q_i^{\sigma-\tau+1}, -q_i^{\sigma+\tau+1},
        q_i^{-\sigma+\tau+1}, -q_i^{-\sigma-\tau+1}, q_i^{2n-4})\\
        \mathsf{DIII}_{2p}: r=p, &\qty(q_i^{\sigma-\tau+1},-q_i^{\sigma+\tau+1}, q_i^{-\sigma+\tau+1},
        -q^{-\sigma-\tau+1}, q_i^4)\\
        \mathsf{DIII}_{2p+1}: r=p, &\qty(q_i^{\sigma-\tau+4}, -q_i^{\sigma+\tau+1}, q_i^{-\sigma+\tau+4}, q_i^{-\sigma-\tau+1}, q_i^4)\\
        \mathsf{EIII}: r=2, &\qty(q_i^{\sigma-\tau+6}, -q_i^{\sigma+\tau+1}, q_i^{-\sigma+\tau+6}, q_i^{-\sigma-\tau+1}, q_i^6)\\
        \mathsf{EVII}: r=3, &\qty(q_i^{\sigma-\tau+1},-q_i^{\sigma+\tau+1},q_i^{-\sigma+\tau+1}, -q_i^{-\sigma-\tau+1},q_i^8)\\
        \mathsf{FII}: r=1,&\qty(q_i^{11}, -q_i^7, q_i^5, -q_i)
    \end{align*}
    where for the non-standard cases we have
    \[
        \bm{t}_i = [\sigma]_{q_i},\qquad
        \bm{s}_i = [\tau]_{q_i}
    \]
    and for the standard cases we have
    \[
        \bm{c}_i\bm{c}_{\tau(i)}=\bm{d}_i\bm{d}_{\tau(i)}=-1,\qquad
        \bm{c}_i=-q_i^\sigma,\qquad \bm{d}_i = -q_i^\tau.
    \]
    More generally, in the non-standard case,
    if $u_1,u_2\in k$ such that $[u_1]_{q_i}=\bm{t}_i$
    and $[u_2]_{q_i}=\bm{s}_i$, then the parameters $\set{a,b,c,d}$ are given
    by
    \[
        \set{-u_1u_2q_i, -u_1^{-1}u_2^{-1}q_i, u_1u_2^{-1}q_i, u_1^{-1}u_2q_i}.
    \]
    This is well-defined: $u_1$ (and independently, $u_2$) can be replaced by $-u_1^{-1}$ (resp. $-u_2^{-1}$), but this just permutes
    $a,b,c,d$, which doesn't affect the polynomials.

    And in the standard cases containing a subdiagram of type $\mathsf{AIV}_m$
    (here, $m=3$ and $5$), the parameters are given by
    \[
        \set{a,b,c,d}=\set{\frac{\bm{c}_i}{\bm{d}_{\tau(i)}}q_i, \frac{\bm{c}_{\tau(i)}}{\bm{d}_i}q_i,\frac{\bm{c}_i}{\bm{d_i}}q_i^{m+1}, \frac{\bm{c}_{\tau(i)}}{\bm{d}_{\tau(i)}}q^{m+1}}.
    \]
\end{remark}

\section{Characters}\label{sec-characters}
With this result in hand, we can proceed to complete the results
from \cite{Mee25}, i.e. the computation of character spherical
functions.

In particular, we shall write
$\Phi^{\mu}_{\chi',\chi}$ for the elementary $(\chi',\chi)$-character spherical
function for the module $L(\mu)$ that is normalised in such a way that
$\Res(\Phi^{\mu}_{\chi',\chi}) = e^\mu + \lot$.

\subsection{Rank 1}\label{sec-char-rank1}
In order to complete Meereboer's computation, we need to consider
the two Hermitean symmetric pairs of rank 1: $\mathsf{AI}_1$ and
$\mathsf{AIV}_n$, and compute their character spherical functions
in full generality.

For $\mathsf{AI}_1$, we simplify notations by considering the case
\[
    \bm{c}_1=\bm{d}_1=q^{-1},\qquad
    \bm{s}_1=[\sigma]_q,\qquad
    \bm{t}_1=[\tau]_q
\]
for $\sigma,\tau\in\Z$, and then generalise using a Zariski argument.
From \cite[Lemma~4.5]{Mee25sym} we see that the integrable characters
of $\uqb$ (resp. $\uqb'$) map $B_1$ to $[\sigma+l]_q$ (resp.
$[\tau+l]_q$) for $l\in\Z$.
Write $\chi_{\sigma,l}: \uqb_{q^{-1},[\sigma]_q}\to k$
for the corresponding character.
Moreover, the bottom element of $X^+(\chi_{\sigma,l})$ is
$\abs{l}\omega_1$.

For $\mathsf{AIV}_n$, the integrable characters of $\uqb$ and $\uqb'$ are given by
\[
    \chi_\ell: \begin{cases}
    K_1K_n^{-1}&\mapsto q^{-\ell},\\
    B_1,B_n&\mapsto 0\\
    x\in \uq_\bullet&\mapsto \epsilon(x)
    \end{cases}
\]
for $\ell\in\Z$.
The bottom element of $X^+(\chi_\ell)$ is $-\ell\omega_1$ or
$\ell\omega_n$ according as $\ell<0$ or $\ell\ge0$.

We start with the easiest non-trivial cases,
i.e. $\chi_{\pm1}$.

\begin{lemma}[$\mathsf{AI}_1$]\label{lem-csf-ai1}
    Let $\epsilon,\delta\in\set{\pm1}$, then
    \[
        \Res(\Phi^{\omega_1}_{\chi_{\tau,\delta},\chi_{\sigma,\epsilon}}) = e^{\omega_1} + \epsilon\delta
        q^{\epsilon\sigma + \delta\tau+1} e^{-\omega_1}.
    \]
\end{lemma}
\begin{proof}
    We start by constructing spherical vectors.
    With respect to the basis $v_{\omega_1},v_{-\omega_1}$, the
    operator $B_1 = F_1 + q^{-1}E_1K_1^{-1} + [\sigma]_qK_1^{-1}$
    has the following symmetric matrix:
    \[
        \mqty([\sigma]_qq^{-1} & 1\\1 & [\sigma]_qq),
    \]
    which has eigenvalues $[\sigma\pm1]_q$ with eigenvector
    $v_{\pm} = v_{\omega_1} \pm q^{\pm\sigma} v_{-\omega_1}$.

    Similarly, replacing $\sigma$ by $\tau$, we obtain the following
    eigenform for the eigenvalue $[\tau\pm1]_q$:
    \[
        f_\pm = v_{\omega_1}^* \pm q^{\pm\tau} v_{-\omega_1}^*.
    \]
    We therefore obtain
    \[
        f_\epsilon(K_h v_\delta) = q^{\langle h,\omega_1\rangle}
        + \epsilon\delta q^{\epsilon\sigma + \delta\tau - \langle h,\omega_1\rangle},
    \]
    so that $\Res(f_\epsilon(\cdot v_\delta)) = q^{-1/2}e^{\omega_1} + \epsilon\delta q^{\epsilon\sigma + \delta\tau+1/2} e^{-\omega_1}$.
    We conclude that a properly normalised character spherical function
    reads
    \[
        e^{\omega_1} + \epsilon\delta q^{\epsilon\sigma+\delta\tau+1}
        e^{-\omega_1}.\qedhere
    \]
\end{proof}

\begin{lemma}[$\mathsf{AIV}_n$]\label{lem-csf-aiv}
    We have
    \[
        \Res(\Phi^{\omega_1}_{\chi_{-1},\chi_{-1}}) = e^{\omega_1} + (-1)^nq\frac{\bm{d}_n}{\bm{c}_1} e^{-\omega_n}.
    \]
    Similarly,
    \[
        \Res(\Phi^{\omega_n}_{\chi_1,\chi_1}) = e^{\omega_n} + (-1)^nq\frac{\bm{d}_1}{\bm{c}_n} e^{-\omega_1},
    \]
\end{lemma}
\begin{proof}
    Let $v\in L(\omega_1)$ transform according to $\chi_{-1}$ of $\uqb$.
    We have $K_{h_i}v=v$ for
    $i=2,\dots,n-1$, so that $v$ lies in the span of
    $v_{\omega_1}$ and $v_{-\omega_n}$,
    say $v=av_{\omega_1}+bv_{-\omega_n}$.
    Note that such a vector $v$ is $\chi_{-1}$-spherical if and only if
    $B_1v=B_nv=0$ ($K_1K_n^{-1}$ acts on both basis vectors as the
    scalar $q$).
    We have
    \begin{align*}
        B_1 v_{\omega_1} &= v_{\omega_2-\omega_1}
        + \bm{c}_1 \ad(E_2\cdots E_{n-1})(E_n)K_1^{-1}v_{\omega_1}
        = v_{\omega_2-\omega_1}\\
        B_n v_{\omega_1} &= 0\\
        B_1 v_{-\omega_n} &= \bm{c}_1 \ad(E_2\cdots E_{n-1})(E_n)K_1^{-1}
        v_{-\omega_n}
        = \bm{c}_1 v_{\omega_2-\omega_1}\\
        B_n v_{-\omega_n} &= \bm{c}_n \ad(E_{n-1}\cdots E_2)(E_1)K_n^{-1}
        v_{-\omega_n} = 0.
    \end{align*}
    We conclude that $v=v_{\omega_1} - \bm{c}_1^{-1}v_{-\omega_n}$ is
    $\chi_{-1}$-spherical.
    Similarly, we have
    \begin{align*}
        v_{\omega_1}^* B_1 &= 0\\
        v_{\omega_1}^* B_n &= \bm{d}_n v^*_{\omega_1}
        \ad(E_{n-1}\cdots E_2)(E_1)K_n^{-1}\\
        &= (-1)^n q^{1-n} \bm{d}_n  v_{\omega_n-\omega_{n-1}}^*\\
        v^*_{-\omega_n}B_1 &= 0\\
        v^*_{-\omega_n}B_n &= v_{\omega_n-\omega_{n-1}}^*.
    \end{align*}
    Consequently, $f=v^*_{\omega_1}- (-1)^n q^{1-n}\bm{d}_n v^*_{-\omega_n}$
    is $\chi_{-1}$-spherical for $\uqb'$.
    Then
    \[
        \Res(f(\cdot v)) = q^{-n/2}e^{\omega_1} + (-1)^n q^{1-n/2}\bm{d}_n\bm{c}_1^{-1}e^{-\omega_n},
    \]
    which implies the first claim.
    We now consider the same for $L(\omega_n)$:
    \begin{align*}
        B_1 v_{\omega_n} &=0\\
        B_n v_{\omega_n} &= v_{\omega_{n-1}-\omega_n}\\
        B_1 v_{-\omega_1} &= 0\\
        B_n v_{-\omega_1} &= \bm{c}_n
        v_{\omega_{n-1}-\omega_n},
    \end{align*}
    which shows that
    $v=v_{\omega_n} - \bm{c}_n^{-1} v_{-\omega_1}$ is
    $\chi_1$-spherical.
    Similarly,
    \begin{align*}
        v_{\omega_n}^*B_1 &= (-1)^n q^{1-n}\bm{d}_1 v^*_{\omega_1-\omega_2}\\
        v_{\omega_n}^*B_n &= 0\\
        v_{-\omega_1}^*B_1 &= v^*_{\omega_1-\omega_2}\\
        v_{-\omega_1}^* B_n &= 0.
    \end{align*}
    Consequently, $f = v^*_{\omega_n}-(-1)^nq^{1-n}\bm{d}_1 v^*_{-\omega_1}$ is $\chi_1$-spherical for $\uqb'$.
    Thus, we obtain
    \[
        \Res(f(\cdot v)) = q^{-n/2}e^{\omega_n}
        + (-1)^n q^{1-n/2}\bm{d}_1\bm{c}_n^{-1} e^{-\omega_1}
    \]
\end{proof}

\begin{lemma}\label{lem-product-csf}
    For $\mathsf{AI}_1$ we have
    \[
        E^{\chi_{\tau,l},\chi_{\sigma,m}} \cdot
        E^{\chi_{\tau+l,l'},\chi_{\sigma+m,m'}}
        \subset E^{\chi_{\tau, l+l'},\chi_{\sigma,m+m'}}
    \]
    and for $\mathsf{AIV}_n$ we have
    \[
        E^{\chi_l,\chi_m}_{\uqb_{\bm{d}},\uqb_{\bm{c}}} \cdot
        E^{\chi_{l'},\chi_{m'}}_{\uqb_{(q^l\bm{d}_1, q^{-l} \bm{d}_n)},
        \uqb_{(q^m\bm{c}_1,q^{-m}\bm{c}_n)}}
        \subset E^{\chi_{l+l'}, \chi_{m+m'}}_{\uqb_{\bm{d}},\uqb_{\bm{c}}}.
    \]
\end{lemma}
\begin{proof}
    This follows from \cite[Lemma~4.1]{Mee25} where we use the following:
    We have
    \[
        \chi_{\sigma+m,m'}(\rho_{\chi_{\sigma,m}}(B_1))
        = \chi_{\sigma+m,m'}(B_1)
        = [\sigma+m+m']_q
        = \chi_{\sigma, m+m'}(B_1),
    \]
    so that we see that $\chi_{\sigma+m,m'}\circ\rho_{\chi_{\sigma,m}}=
    \chi_{\sigma,m+m'}$.
    Furthermore,
    \begin{align*}
        \chi_{m'}(\rho_{\chi_m}(K_1K_n^{-1}))
        &= \chi_{m'}(q^{-m} K_1K_n^{-1})
        = q^{-m-m'} K_1K_n^{-1}\\
        &= \chi_{m+m'}(K_1K_n^{-1}),
    \end{align*}
    which shows that $\chi_{m'}\circ\rho_{\chi_m}=\chi_{m+m'}$.
\end{proof}

\begin{proposition}\label{prop-radial-parts-rk1-csf}
    For $\mathsf{AI}_1$,
    let $l,m\in\Z$ have the same parity.
    Define $a_\pm := \frac{l\pm m}{2}\in\Z$,
    then
    \begin{align*}
        \Res(\Phi^{(\abs{a_+}+\abs{a_-})\omega_1}_{\chi_{\tau,l},\chi_{\sigma,m}}) =
        e^{(\abs{a_+}+\abs{a_-})\omega_1}
        &\qty(-q^{\sigma+\tau+1}e^{-\alpha_1};q^2)_{(a_+)^+}\\
        \cdot&\qty(-q^{-\sigma-\tau+1}e^{-\alpha_1};q^2)_{(-a_+)^+}\\
        \cdot &\qty(q^{\tau-\sigma+1}e^{-\alpha_1};q^2)_{(a_-)^+}\\
        \cdot&\qty(q^{\sigma-\tau+1}e^{-\alpha_1};q^2)_{(-a_-)^+}.
    \end{align*}
    For $\mathsf{AIV}_n$ and $l>0$ we have
    \begin{align*}
        \Res(\Phi^{l\omega_1}_{\chi_{-l},\chi_{-l}}) &= e^{l\omega_1}
        \qty((-1)^{n-1}q\frac{\bm{c}_n}{\bm{d}_1}e^{-2\tilde{\alpha}_1};q^2)_l\\
        \Res(\Phi^{l\omega_n}_{\chi_l,\chi_l}) &= e^{l\omega_n}
        \qty((-1)^{n-1}q\frac{\bm{c}_1}{\bm{d}_n}e^{-2\tilde{\alpha}_1};q^2)_l.
    \end{align*}
\end{proposition}
\begin{proof}
    We start with $\mathsf{AI}_1$.
    We can write $(l,m)$ as follows as a sum of minimal length
    of our generators $\set{\pm1}^2$:
    \[
        (l,m) = (a_+)^+(1,1)
        + (-a_+)^+(-1,-1)
        + (a_-)^+(1,-1)
        + (-a_-)^+(-1,1).
    \]
    Thus, Lemma~\ref{lem-product-csf} implies that
    \begin{align}
        &E^{\chi_{\tau, (a_+)^+},\chi_{\sigma, (a_+)^+}}
        \cdot E^{\chi_{\tau+(a_+)^+, -(-a_+)^+},
        \chi_{\sigma+(a_+)^+,-(-a_+)^+}}\nonumber\\
        &\cdot
        E^{\chi_{\tau + a_+, (a_-)^+},
        \chi_{\sigma + a_-, -(a_-)^+}}
        \cdot
        E^{\chi_{\tau + a_+ + (a_-)^+, -(-a_-)^+},
        \chi_{\sigma + a_+ - (a_-)^+, (a_-)^+}}\nonumber\\
        \subset &E^{\chi_{\tau, l}, \chi_{\sigma,m}},\label{eq-product-of-csf-ai1}
    \end{align}
    so in order to construct an element of $E^{\chi_{\tau,l},\chi_{\sigma,m}}$, we shall first construct character spherical
    functions for which $\abs{l}=\abs{m}$.

    Let $\epsilon,\delta\in\set{\pm1}$ and $l>0$, then
    \[
        \prod_{i=0}^{l-1} E^{\chi_{\tau+\delta i, \delta},\chi_{\sigma+\epsilon i,\epsilon}} \subset
        E^{\chi_{\tau,\delta l},\chi_{\sigma,\epsilon l}},
    \]
    so that there is an element $\Phi\in E^{\chi_{\tau,\delta l},\chi_{\sigma,\epsilon l}}$ with
    \[
        \Res(\Phi) = \prod_{i=0}^{l-1} \qty(e^{\omega_1} + \epsilon\delta
        q^{1 + 2i + \epsilon\sigma + \delta\tau} e^{-\omega_1})
        = e^{l\omega_1} \qty(-\epsilon\delta q^{\epsilon\sigma+\delta\tau + 1}e^{-\alpha_1};q^2)_l.
    \]
    Inserting this into \eqref{eq-product-of-csf-ai1} we obtain that there is $\Phi\in E^{\chi_{\tau,l},\chi_{\sigma,m}}$
    satisfying
    \begin{align*}
        \Res(\Phi) = e^{(\abs{a_+}+\abs{a_-})\omega_1}
        &\qty(-q^{\sigma+\tau+1}e^{-\alpha_1};q^2)_{(a_+)^+}\\
        \cdot&\qty(-q^{-(\sigma+(a_+)^+)-(\tau-(a_+)^+)+1}e^{-\alpha_1};q^2)_{(-a_+)^+}\\
        \cdot &\qty(q^{\tau+a_+-(\sigma+a_+)+1}e^{-\alpha_1};q^2)_{(a_-)^+}\\
        \cdot&\qty(q^{\sigma + a_+-(a_-)^+-\tau-a_+-(a_-)^++1}e^{-\alpha_1};q^2)_{(-a_-)^+}\\
        = e^{(\abs{a_+}+\abs{a_-})\omega_1}
        &\qty(-q^{\sigma+\tau+1}e^{-\alpha_1};q^2)_{(a_+)^+}
        \qty(-q^{-\sigma-\tau+1}e^{-\alpha_1};q^2)_{(-a_+)^+}\\
        \cdot &\qty(q^{\tau-\sigma+1}e^{-\alpha_1};q^2)_{(a_-)^+}
        \qty(q^{\sigma-\tau+1}e^{-\alpha_1};q^2)_{(-a_-)^+}
    \end{align*}
    since only one of $(\pm a_+)^+$ can be nonzero at a time.

    By construction, this spherical function contains matrix elements
    of $L(\omega_1)^{\otimes (\abs{a_+}+\abs{a_-})}$,
    in whose decomposition there is exactly one module (with multiplicity one)
    that also lies in $X^+(\chi_{\sigma,m})\cap X^+(\chi_{\tau,m})$,
    namely $(\abs{a_+}+\abs{a_-})\omega_1=\max(\abs{l},\abs{m})\omega_1$,
    which is the bottom element.
    Consequently, $\Phi\in E^{\chi_{\tau,l},\chi_{\sigma,m}}((\abs{a_+}+\abs{a_-})\omega_1)$, and more specifically
    $\Phi = \Phi^{(\abs{a_+}+\abs{a_-})\omega_1}_{\chi_{\tau,l},\chi_{\sigma,m}}$.

    Similar arguments and Lemma~\ref{lem-csf-aiv} show that for $\mathsf{AIV}_n$ we obtain
    \begin{align*}
        \Res(\Phi^{l\omega_1}_{\chi_{-l},\chi_{-l}})&= \prod_{i=0}^{l-1}
        \qty(e^{\omega_1} + (-1)^nq^{1+2i}\frac{\bm{d}_n}{\bm{c}_1} e^{-\omega_n})\\
        &= e^{l\omega_1}
        \qty((-1)^{n-1}q\frac{\bm{c}_n}{\bm{d}_1}e^{-2\tilde{\alpha}_1};q^2)_l\\
        \Res(\Phi^{l\omega_n}_{\chi_l,\chi_l}) &= \prod_{i=0}^{l-1}
        \qty(e^{\omega_n} + (-1)^nq^{1+2i}\frac{\bm{d}_1}{\bm{c}_n} e^{-\omega_1})\\
        &= e^{l\omega_n}\qty((-1)^{n-1}q\frac{\bm{c}_1}{\bm{d}_n}
        e^{-2\tilde{\alpha}_1};q^2)_l,
    \end{align*}
    using that $\omega_1+\omega_n = 2\tilde{\alpha}_1$ and that
    $\bm{c}_1\bm{c}_n=\bm{d}_1\bm{d}_n$.
\end{proof}

As before we write $\Sigma^*$ for the root system of type
$\mathsf{BC}_1$ containing $\Sigma$ (as long roots),
i.e.
\[
    \Sigma^* = \begin{cases}
        \set{\pm\frac{1}{2}\alpha_1,\pm\alpha_1} & \mathsf{AI}_1\\
        \set{\pm\tilde{\alpha}_1, \pm2\tilde{\alpha}_1} & \mathsf{AIV}_n
    \end{cases}
\]
and write $\Sigma^*_\mathrm{l}$ for the set of long roots.

\begin{corollary}
    Write $\chi_l$ for the $l$-th character of $\uqb'$ (also in case $\bm{d}_1$ is not a $q$-number) and
    $\chi_m$ for the $m$-th character of $\uq$, and
    let $\mathfrak{B}^+(\chi_l,\chi_m)=\set{\mu}$.
    Then
    \[
        \Res(\Phi^{\mu}_{\chi_l,\chi_m})\overline{\Res(\Phi^\mu_{\chi_l,\chi_m})}
        = \prod_{\alpha\in \Sigma^*_\mathrm{l}}
        \frac{f_\alpha(a,b,c,d)}{f_\alpha\qty(aq_\alpha^{(l-m)^+},bq_\alpha^{(l+m)^+}, cq_\alpha^{(m-l)^+},dq_\alpha^{(-l-m)^+})}
    \]
    where
    \[
        f_\alpha(a,b,c,d) = \qty(ae^\alpha,be^\alpha,ce^\alpha,de^\alpha;q_\alpha^2)_\infty
    \]
    and where
    $(a,b,c,d)$ are the Askey--Wilson parameters from Lemma~\ref{lem-rk1-nonletzter-macdonald} (and where $l=m$ for $\mathsf{AIV}_n$).
\end{corollary}
\begin{proof}
    This follows from Proposition~\ref{prop-radial-parts-rk1-csf}.

    For $\mathsf{AI}_1$, we note that the LHS and the RHS are both
    rational functions in $\bm{s}_1,\bm{t}_1$, and loc.cit.
    establishes equality whenever $\bm{s}_1=[\sigma]_q$ and
    $\bm{t}_1=[\tau]_q$.
    (Note that $\alpha = \pm\tilde{\alpha}_1$ and we have
    $q_\alpha^2=q^{\abs{\tilde{\alpha}_1}^2}=q^{\abs{\alpha_1}^2}=q^2$.)

    For $\mathsf{AIV}_n$, we note that
    $b=(-1)^{n-1}\frac{\bm{c}_1}{\bm{d}_n}q_\alpha$ and that
    $d=(-1)^{n-1}\frac{\bm{c}_n}{\bm{d}_1}q_\alpha$.
\end{proof}

\subsection{Higher Rank}\label{sec-char-higher-rank}
We now return to the general case and adopt notation from
Section~\ref{sec-rk1-reduction}.

For $\beta\in\Sigma^*$ define
\[
    w_\alpha := \frac{f_\alpha(a,b,c,d)}{f_\alpha\qty(aq_\alpha^{(l-m)^+},bq_\alpha^{(l+m)^+}, cq_\alpha^{(m-l)^+},dq_\alpha^{(-l-m)^+})}
    \in k[e^\alpha]
\]
and define
\[
    w := \prod_{\alpha\in \Sigma^*_{\mathrm{l}}}w_\alpha,
\]
where $\Sigma^*_{\mathrm{l}}$ are the long roots.

\begin{lemma}\label{lem-shape-bottom}
    Let $i\in I_\circ$ be such that $\Q\tilde{\alpha}_i\cap\Sigma^*_{\mathrm{l}}=\set{\beta}$, and let
    $\chi'$ be a character of $\uqb'$ (and $\chi$ of $\uqb$) such
    that $\chi'_{\uqb'[i]} = \chi_l$ and $\chi|_{\uqb[i]} = \chi_m$,
    and let
    $\mu\in\mathfrak{B}^+(\chi',\chi)$,
    then
    \[
        \Res(\Phi^\mu_{\chi',\chi})\overline{\Res(\Phi^\mu_{\chi',\chi})}
        \propto w.
    \]
\end{lemma}
\begin{proof}
    Write $w_1:= \Res(\Phi^\mu_{\chi',\chi})\overline{\Res(\Phi^\mu_{\chi',\chi})}$,
    and write $w_1w_2=q$ for a rational function $w_2\in k(P(2\Sigma^*))$.
    By \cite[Lemma~5.6]{Mee25}, we can write $w_1=w_\beta w_{-\beta} k$ for $k\in k[e^\alpha\mid\alpha\in P(2\Sigma^*),\alpha\cdot\beta=0]$
    (as $w_\beta$ is the shifted character spherical function
    for $(\uq[i],\uqb[i], \uqb'[i])$).

    We thus obtain $w=w_1w_2=w_\beta w_{-\beta} k w_2$, which implies
    $\frac{w}{w_\beta w_{-\beta}} = kw_2$.
    Note that for $\alpha\in\Sigma^*_{\mathrm{l}}$ we have either
    $\alpha=\pm\beta$ or $\alpha\cdot\beta=0$, so that
    \[
        \frac{w}{w_\beta w_{-\beta}}
        = \prod_{\substack{\alpha\in\Sigma^*_{\mathrm{l}}\\\alpha\cdot\beta=0}} w_\alpha
        \in k[e^\alpha\mid \alpha\in P(2\Sigma^*), \alpha\cdot\beta=0].
    \]
    We conclude that $w_2$, too, can be written in terms of monomials
    orthogonal to $\beta$.
    
    Moreover, $w_1=\Res(\Xi(\Phi^\mu_{\chi',\chi},\Phi^\mu_{\chi',\chi}))$ is the shifted restriction of the zonal
    spherical function $\Xi(\Phi^\mu_{\chi',\chi},\Phi^\mu_{\chi',\chi})$ (\cite[Definition~2.16]{qmsf}
    or rather \cite[Lemma~4.5]{Mee25}).
    Consequently, $w_1$ is
    $W^\Sigma$-invariant, as is $w$, so that we conclude that
    $w_2$ must be $W^\Sigma$-invariant as well.
    But since it can be written in terms of monomials orthogonal
    to $\beta$, the invariance implies that it is constant.
\end{proof}

\begin{theorem}\label{thm-chisf}
    Let $(a,b,c,d,t)$ or $(a,b,c,d)$ be the parameters from 
    Theorem~\ref{thm-zsf}.
    Let $i\in I_\circ$ be such that $\Q\tilde{\alpha}_i\cap\Sigma^*_{\mathrm{l}}=\set{\beta}$, and let
    $\chi'$ be a character of $\uqb'$ (and $\chi$ of $\uqb$) such
    that $\chi'_{\uqb'[i]} = \chi_l$ and $\chi|_{\uqb[i]} = \chi_m$.
    Let $\mathfrak{B}(\chi',\chi)=\set{\mu}$ and
    $\lambda\in X^+(\chi',\chi)$, then
    \[
        \Res(\Phi^{\lambda}_{\chi',\chi})
        = \Res(\Phi^{\mu}_{\chi',\chi})
        P_{\lambda-\mu},
    \]
    where $P_{\lambda-\mu}$ is the symmetric Macdonald polynomial
    with parameters
    \[
        (a',b',c',d',t)=(aq_\beta^{(l-m)^+}, bq^{(l+m)^+}, cq^{(m-l)^+},
        dq^{(-l-m)^+},t)
    \]
    and all the other data as in Theorem~\ref{thm-zsf}.
\end{theorem}
\begin{proof}
    By \cite[Lemma~6.5]{qmsf}, we have that
    $\Res(\Phi^{\lambda}_{\chi',\chi})$ is divisible by
    $\Res(\Phi^{\mu}_{\chi',\chi})$.
    Furthermore, the quotients is a $W^\Sigma$-symmetric
    polynomial $Q_\lambda$ with leading term $e^{\lambda-\mu}$.

    Write $\triangledown_{a,b,c,d,t}$ for the Macdonald weight distribution for the parameters $(a,b,c,d,t)$.
    Then
    \[
        w\triangledown_{a,b,c,d,t} = \triangledown_{a',b',c',d',t},
    \]
    so that
    \begin{align*}
        h(\Xi(\Phi^\lambda_{\chi',\chi},\Phi^{\lambda'}_{\chi',\chi}))
        &= \ct(\Res(\Phi^{\lambda}_{\chi',\chi})
        \overline{\Res(\Phi^{\lambda'}_{\chi',\chi})
        }\triangledown_{a,b,c,d,t})\\
        &= \ct(Q_\lambda\overline{Q_\lambda}
        \Res(\Phi^\mu_{\chi',\chi})\overline{\Res(\Phi^\mu_{\chi',\chi})}\triangledown_{a,b,c,d,t})\\
        &= \ct(Q_\lambda\overline{Q_\lambda} w\triangledown_{a,b,c,d,t})\\
        &= \ct(Q_\lambda\overline{Q_\lambda}\triangledown_{a',b',c',d',t})\\
        &\in \delta_{\lambda,\lambda'}k.
    \end{align*}
    This shows that $Q_\lambda = P_{\lambda-\mu}$ is a Macdonald polynomial with parameters $(a',b',c',d',t)$.
\end{proof}

\subsection{\texorpdfstring{$\epsilon$}{Epsilon}-Symmetry}
We now consider a special case in which we can not only give
an interpretation for the fraction
$\frac{\Res(\Phi^\lambda_{\chi',\chi})}{\Res(\Phi^\mu_{\chi',\chi})}$
(for $\lambda\in X^+(\chi',\chi)$ and $\mu\in \mathfrak{B}^+(\chi',\chi)$),
but for the entire polynomial $\Res(\Phi^\lambda_{\chi',\chi})$.

\begin{lemma}\label{lem-bottom-in-the-wild}
    Let $(\uq,\uqb,\uqb')$ be a Hermitean symmetric pair with
    $\mathcal{S}=\set{i}$ (i.e. with one non-standard parameter).
    Fix $\bm{t}_i=[u_1]_{q_i}$ and $\bm{s}_i=[u_2]_{q_i}$.
    Let $\abs{l\pm m}=2$, set $\chi':= \chi_{u_1,l}$ and
    $\chi:= \chi_{u_2,m}$ and consider $\mu\in\mathfrak{B}^+(\chi',\chi)$.
    In this case we have
    \[
        \Res(\Phi^\mu_{\chi', \chi})
        = \prod_{\alpha\in\Sigma^{*,+}_{\mathrm{l}}}
        \qty(e^\alpha + xye^{-\alpha} - x - y)
    \]
    (note that $2\mu=\sum_{\alpha\in\Sigma^{*,+}_{\mathrm{l}}}\alpha$
    for
    \[
        \set{x,y} = \begin{cases}
            \set{a,b} & (l,m)=(0,2)\\
            \set{b,c} & (l,m)=(2,0)\\
            \set{c,d} & (l,m)=(0,-2)\\
            \set{a,d} & (l,m)=(-2,0)
        \end{cases}
    \]
    for $(a,b,c,d)$ as in Theorem~\ref{thm-zsf}.
\end{lemma}
\begin{proof}
    This follows from Proposition~\ref{prop-radial-parts-rk1-csf}
    and a proof analogous to that of Lemma~\ref{lem-shape-bottom}.
\end{proof}

\begin{proposition}\label{prop-epsilon-koornwinder}
    In the setting of Lemma~\ref{lem-bottom-in-the-wild} let
    $\epsilon: W_\Sigma\to k^\times$ map $r_i = s_{\tilde{\alpha}_i}$
    to $-1$ and all other generators to $1$.
    For every $\lambda\in X^+(\chi',\chi)$ we have
    \[
        \Res(\Phi^\lambda_{\chi',\chi})
        \propto P^{(\epsilon)}_\lambda
    \]
    (the $\epsilon$-symmetric Macdonald polynomial)
    for the initial data
    \[
        (2R_0,-2\Delta_0,\mathrm{t},2L,2L)
    \]
    (see Theorem~\ref{thm-zsf}) and multiplicity function 
    $\kappa$ satisfying
    \[
        (\kappa_\theta \kappa_{2\theta}, -\kappa_\theta \kappa_{2\theta}^{-1}, q_i\kappa_{a_0}\kappa_{2a_0},
        -q_i \kappa_{a_0}\kappa_{2a_0}^{-1})
        = \begin{cases}
            (a,b,c,d)  & (l,m) = (0,2)\\
            (c,b,a,d) & (l,m) = (2,0)\\
            (c,d,a,b) & (l,m) = (0,-2)\\
            (a,d,c,b) & (l,m) = (-2,0)
        \end{cases}
    \]
    and $\kappa_{a_j}^2=t$ for every simple medium root $a_j$
    for $(a,b,c,d,t)$ as in loc.cit.
\end{proposition}
\begin{proof}
    By Theorem~\ref{thm-chisf} we have
    \[
        \Res(\Phi^\lambda_{\chi',\chi})
        = \Res(\Phi^\mu_{\chi',\chi})
        P_{\lambda-\mu}',
    \]
    where $P_{\lambda-\mu}'$ is the symmetric Macdonald polynomial
    with the parameters $\set{a',b',c',d'}$ from loc.cit.
    Note that these parameters arise from $(a,b,c,d)$ from
    Theorem~\ref{thm-zsf} by multiplying two of them with $q_i^2$.
    In the four cases from the statement, these are
    $(a,b)$, $(c,b)$, $(c, d)$, $(a,d)$ respectively.
    Note that the symmetric Macdonald polynomials are symmetric in
    arbitrary permutations of $a,b,c,d$, so we can always permute the
    parameters such that when passing from zonal to character
    spherical functions, it is the first two parameters that
    are multiplied by $q_i^2$.

    This yields the permutations from the claim.
    If we let $\kappa,\kappa'$ correspond to unprimed and primed
    parameters, then we see that $\kappa'_\theta=q_i^2\kappa_\theta$,
    and $\kappa'_a=\kappa_a$ for $a$ from any other orbit.
    
    Write $\ell$ for the function on the affine root system associated
    to our Macdonald initial data that maps $\theta\mapsto 1$ and
    all other orbits to 0, then $\kappa'=\kappa q_i^{2\ell}$
    and hence
    \[
        P_{\lambda-\mu}'=P_{\lambda-\mu, \kappa q_i^{2\ell}}.
    \]
    Note that for $a$ indivisible, we have $\ell(a)=1$ iff $a$ 
    lies in the same orbit as $\alpha\in2\Sigma^*$ with $\epsilon(r_\alpha)=-1$, so $\ell$ and $\epsilon$ are related
    as in \cite[\S5.8]{Mac03}.

    Moreover, note by Lemma~\ref{lem-bottom-in-the-wild} that
    $\rho_\ell=\mu$, and 
    $\Res(\Phi^\mu_{\chi',\chi})\propto\delta_\epsilon$ 
    (as in \cite[\S5.8.3]{Mac03}; note that choosing the opposing
    positive subsystem maps $\delta_\epsilon\mapsto\overline{\delta_\epsilon}$).
    By \cite[\S5.8.9]{Mac03} we conclude the claim.
\end{proof}

\begin{remark}
    Note that the symmetric Koornwinder polynomials depend on the
    multiplicity function $\kappa$ only through the five parameters
    $(a,b,c,d,t)$, the first four of which can be permuted freely.
    Therefore, when dealing with only those polynomials,
    the order of $a,b,c,d$ does not matter when deriving $\kappa$ from $a,b,c,d$ and we could have picked any order in
    Theorems~\ref{thm-zsf} and \ref{thm-chisf}.

    When also considering arbitrary $\epsilon$-symmetric Koornwinder
    polynomials, the symmetry group is smaller: it changes from
    $S_4$ to $S_2\times S_2$.
    Fixing a function $\delta_\epsilon$ (for the $\epsilon$ we consider in Proposition~\ref{prop-epsilon-koornwinder}) corresponds
    to making a choice of subset $\set{a,b}$ of $\set{a,b,c,d}$.

    This choice arises from the fact that picking one of the four choices for $(l,m)$ satisfying $\abs{l\pm m}=2$
    corresponds to picking two parameters that change when passing
    from the zonal to the $(\chi',\chi)$-spherical function.
    It is for this reason that we really need to consider specific
    permutations of $\set{a,b,c,d}$ to derive $\kappa$ from
    in Proposition~\ref{prop-epsilon-koornwinder}.
\end{remark}

\subsection{Spherical Pair}
Another special interpretation can be given for the character spherical functions
for those cases that have $\mathsf{AIV}$ as a subdiagram:
$\mathsf{AIII}_{n,p}$ ($n>2p-1$), $\mathsf{AIV}_n$, $\mathsf{DIII}_{2p+1}$,
and $\mathsf{EIII}$.

\begin{lemma}\label{lem-sph-lie-alg}
    Let $(\mathfrak{g},\mathfrak{h})$ be the corresponding symmetric pair
    of Lie algebras.
    Then, $(\mathfrak{g},[\mathfrak{h},\mathfrak{h}])$ is a spherical pair.
\end{lemma}
\begin{proof}
    Follows from \cite[Tabelle~1]{Kr79}.
\end{proof}

\begin{definition}
    Let $i\in I_\circ$ be such that $I[i]$ is of type $\mathsf{AIV}_m$ (i.e.
    one of the two nodes that are touching black nodes).
    Let
    \[
	Y' := \set{h\in Y^\Theta\where \langle h,\omega_i\rangle=0}.
    \]
    Define $\mathbf{C},\mathbf{C}'$ to be those subalgebras of $\uqb,\uqb'$
    generated by
    \[
	U_\bullet,\qquad K_h \quad (h\in Y'),\qquad B_i \quad (i\in I_\circ).
    \]
    These are the quantisations of $[\mathfrak{h},\mathfrak{h}]$ in the
    notation of Lemma~\ref{lem-sph-lie-alg}.
\end{definition}

\begin{lemma}\label{lem-all-chars-equal-C}
    For every character $\chi$ of $\uqb$ we have $\chi|_{\mathbf{C}}=\epsilon|_{\mathbf{C}}$.
\end{lemma}
\begin{proof}
    By \cite[Lemma~4.16]{Mee25sym}, there exists $\lambda\in X^+$ such that
    $\chi(B_j)=0$ ($j\in I_\circ$), $\chi(E_j)=\chi(F_j)=0$ ($j\in I$),
    and $\chi(K_h)=q^{\langle h,\lambda\rangle}$.

    From Section~\ref{sec-char-rank1} we know that $\lambda$ can be chosen to
    be a multiple of $\omega_i$ or $\omega_{\tau(i)}$.
    Note that $\Theta(\omega_{\tau(i)})=-\omega_i$, so that every element of
    $Y'$ is orthogonal to $\lambda$.
    As a consequence, we have $\chi(K_h)=1$ for all $K_h\in\mathbf{C}$,
    and hence the claim follows.
\end{proof} 

\begin{lemma}
    In the specialisation procedure outlined in \cite[\S 10]{Kol14},
    $\mathbf{C}$ specialises to $[\mathfrak{h},\mathfrak{h}]$.
\end{lemma}
\begin{proof}
    As a subalgebra of $\mathfrak{h}$, its semisimple part $[\mathfrak{h},\mathfrak{h}]$ is characterised as the kernel of all characters.
    Note that a sublattice of these characters can be obtained by specialising
    the characters of $\uqb$.
    In particular, by Lemma~\ref{lem-all-chars-equal-C}, the specialisation
    of $\mathbf{C}$ is contained in the universal enveloping algebra of
    $\comm{\mathfrak{h}}{\mathfrak{h}}$.

    Conversely, we note in analogy to \cite[Lemma~2.8]{Kol14} that
    $\comm{\mathfrak{h}}{\mathfrak{h}}$ is generated by
    $e_j, f_j$ ($j\in I_\bullet$), $f_j+\theta(f_j)$ ($j\in I_\circ$),
    and the $\C$-span of $Y'$.
    Conveniently, $E_j, F_j$ ($j\in I_\bullet$), $B_j$ ($j\in I_\circ$), 
    and $\frac{K_h-1}{q-1}$ ($h\in Y'$) specialise to
    these particular elements.
\end{proof}

\begin{remark}
    Note that $\mathbf{C}$ fails to be a right coideal.
    In fact, any subalgebra $\mathbf{C}\subset\uqb$ on which all of $\uqb$'s
    characters are trivial and which contains $B_i$, fails to be a right coideal.
    
    Assume $\mathbf{C}\subset\uqb$ is a coideal subalgebra on which all
    characters of $\uqb$ are trivial.
    Consider the shift of base point $x\mapsto \chi(x_{(1)})x_{(2)}$
    (cf. \cite[(2.6)]{Mee25}), which maps $B_i$ with parameter $\bm{c}_i$ to
    $B_i$ with parameter $\chi(K_{-\alpha_i-\Theta(\alpha_i)})\bm{c}_i$,
    which lies outside of $\uqb$ and hence also of $\mathbf{C}$.
    At the same time we have $\Delta(\mathbf{C})\subset\mathbf{C}\otimes\uq$
    and $\chi|_{\mathbf{C}}=\epsilon$, so that
    \[
        \chi(x_{(1)}) x_{(2)} = \epsilon(x_{(1)})x_{(2)} = x
    \]
    for $x\in\mathbf{C}$.
    As a consequence, $\mathbf{C}$ cannot contain $B_i$.
\end{remark}

Despite $\mathbf{C},\mathbf{C}'$ not being coideal subalgebras of $\uq$, we
can still wonder what its zonal spherical functions are.
\begin{proposition}\label{prop-C-zsf}
    The zonal spherical functions for
    $(\uq,\mathbf{C},\mathbf{C}')$ are precisely the character spherical
    functions for $(\uq,\uqb,\uqb')$.
\end{proposition}
\begin{proof}
    Let $f\in E^{\chi',\chi}_{\uqb',\uqb}$, let $c'\in\mathbf{C}',c\in\mathbf{C}$,
    and $x\in \uq$, then
    \[
        f(c'xc) = \chi'(c') f(x) \chi(c) = \epsilon(c') f(x) \epsilon(c)
    \]
    by Lemma~\ref{lem-all-chars-equal-C} applied to $\mathbf{C}$ and $\mathbf{C}'$,
    i.e. $f\in E^{\epsilon,\epsilon}_{\mathbf{C}',\mathbf{C}}$.

    Conversely, let $M$ be a finite-dimensional $\uq$-module
    and let $v\in M$ be $\mathbf{C}$-fixed.
    Let $h\in Y^\Theta$, then $K_h$ normalises $\mathbf{C}$, and hence
    maps $v$ to another $\mathbf{C}$-fixed vector.
    Since $M$ is integrable, all $K_h$ ($h\in Y^\Theta$)
    are simultaneously diagonalisable on the $\mathbf{C}$-fixed vectors. 
    Hence, every $\mathbf{C}$-fixed
    vector can be decomposed as a sum of $\uqb$-character-spherical vectors
    (i.e. into a sum of $v_\chi\in M$, for a character $\chi$ of $\uqb$,
    satisfying $xv_\chi=\chi(x)v_\chi$ for $x\in\uqb$).

    A similar statement holds for $\mathbf{C}'$-fixed elements of
    $M^*$.

    We thus conclude that every zonal spherical function
    for $(\uq,\mathbf{C},\mathbf{C}')$ can be decomposed as a sum of
    character spherical functions for $(\uq,\uqb,\uqb')$.
\end{proof}

\begin{remark}
    Since $\mathbf{C}$ and $\mathbf{C}'$ are no coideals, the multiplication
    on $\mathcal{A}$ does not restrict to a multiplication on
    $E^{\epsilon,\epsilon}_{\mathbf{C}',\mathbf{C}}$.
    One might hope that the relation from \cite[Lemma~4.1]{Mee25}
    can be combined with an appropriate twist, but alas, comparing
    \cite[(2.7)]{Mee25} with \cite[(31)]{qmsf}, we find that
    $\rho_\chi(\uqb)$ is not related to $\uqb$ via a conjugation automorphism 
    from $\tilde{H}$ (cf. \cite[\S3.2]{Kol14}).
\end{remark}

\section{Macdonald Polynomials in the Wild}\label{sec-classification-MK}
From Section~\ref{sec-char-higher-rank} we can now overall conclude 
that the following Koornwinder polynomials can be realised as
character spherical functions for a non-standard case QSP:
\begin{enumerate}
    \item In Rank 1 with parameters $(a,b,c,d)$ where
    $ac,bd\in q^\N$.
    \item In Rank 2 with parameters $(a,b,c,d,t)$ for
    $ac,bd\in q^\N$ and $t\in q^{\N/2}$.
    \item In Rank 3 with parameters $(a,b,c,d,t)$ for
    $ac, bd\in q^{\N}$ and $t\in \set{q^{1/2},q,q^2,q^4}$.
    \item In higher rank with parameters $(a,b,c,d,t)$ for
    $ac,bd\in q^{\N}$ and $t\in \set{q^{1/2},q,q^2}$.
\end{enumerate}

We let $ac=:q^{1+a_-}$ and $bd=:q^{1+a_+}$ for $a_\pm\in\N_0$, and
\[
    m:= a_++a_-,\qquad l:= a_+-a_-,
\]
integers of the same parity.
Moreover, pick $u_1,u_2$ such that
\[
    c = q^{1/2}u_1^{-1}u_2,\qquad
    d = -q^{1/2}u_1^{-1}u_2^{-1}.
\]
Consider $(\uq,\uqb,\uqb')$ of the following Satake types:
\begin{itemize}
    \item In rank 1: $\mathsf{AI}_1$ with $i:=1$.
    \item In rank 2 if $t$ is an integral power of $q$, write
    $t=q^{n-2}$.
    If $n\ge4$: $\mathsf{DI}_{n,2}$ with $i:=1$.
    If $n=3$: $\mathsf{AIII}_{3,2}$ with $i:=2$.
    \item In rank 2 if $t$ is a half-integer power of $q$, write
    $t=q^{n-\frac{3}{2}}$. Then: $\mathsf{BI}_{n,2}$ with $i:=1$
    \item In rank 3 if $t=q^4$: $\mathsf{EVII}$ with $i:=7$.
    \item In rank $n\ge3$ if $t=q^{1/2}$: $\mathsf{CI}_n$ with $i:=n$.
    \item In rank $n\ge3$ if $t=q$: $\mathsf{AIII}_{2n-1,n}$ with
    $i:=n$.
    \item In rank $n\ge3$ if $t=q^2$: $\mathsf{DIII}_{2n}$ with
    $i:= 2n$.
\end{itemize}
Fix the inner product on $X$ to satisfy $\abs{\alpha_i}^2=\abs{\tilde{\alpha}_i}^2=1$
(so that $q_i=q^{1/2}$),
and let
\[
    (\bm{c}_j,\bm{s}_j,\bm{d}_j,\bm{t}_j)= \begin{cases}
        (q^{-1},[u_2]_{q^{1/2}}, q^{-1}, [u_1]_{q^{1/2}}) &
        j=i\\
        (1,0,1,0) & j\ne i.
    \end{cases}
\]
Then the $(\chi_{\bm{t},l},\chi_{\bm{s},m})$-spherical functions
restrict to a multiple of the desired Koornwinder polynomials
(i.e. multiplied by the respective bottom function).

Beyond the 1-dimensional $(\uqb',\uqb)$-types, symmetric Koornwinder--Macdonald polynomials
can always be obtained if $\#\mathfrak{B}^+(\gamma',\gamma)=1$
(for a suitable simple $\uqb'$-module $\gamma'$ and
$\uqb$-module $\gamma$).
These cases are abundant.
For examples we refer to \cite{EK94, OS05}, and to \cite[\S7.1]{qmsf}.

\printbibliography
\end{document}